\documentclass{amsart}

\usepackage{amsmath,amssymb}
\usepackage{color}
\usepackage{empheq}
\usepackage{CJK}
\usepackage{hyperref}

\newlength\mytemplen
\newsavebox\mytempbox
\definecolor{myblue}{rgb}{.97,.97,1}

\makeatletter
\newcommand\mybluebox{%
    \@ifnextchar[%]
       {\@mybluebox}%
       {\@mybluebox[0pt]}}

\def\@mybluebox[#1]{%
    \@ifnextchar[%]
       {\@@mybluebox[#1]}%
       {\@@mybluebox[#1][0pt]}}

\def\@@mybluebox[#1][#2]#3{
    \sbox\mytempbox{#3}%
    \mytemplen\ht\mytempbox
    \advance\mytemplen #1\relax
    \ht\mytempbox\mytemplen
    \mytemplen\dp\mytempbox
    \advance\mytemplen #2\relax
    \dp\mytempbox\mytemplen
    \colorbox{myblue}{\hspace{1em}\usebox{\mytempbox}\hspace{1em}}}
\makeatother

\def\CP{\mathcal{P}}

\def\frakg{\ensuremath{\mathfrak{g}}}
\def\frakh{\ensuremath{\mathfrak{h}}}

\def\frakp{\ensuremath{\mathfrak{p}}}
\def\frakU{\ensuremath{\mathfrak{U}}}

\newcommand{\bC}{\ensuremath{\mathbb{C}}}

\newcommand{\bZ}{\ensuremath{\mathbb{Z}}}

\newcommand{\bN}{\ensuremath{\mathbb{N}}}

\newcommand{\beq}{\begin{equation*}\begin{aligned}}
\newcommand{\eeq}{\end{aligned}\end{equation*}}

\newcommand{\beqn}{\begin{equation}\begin{aligned}}
\newcommand{\eeqn}{\end{aligned}\end{equation}}

\newcommand{\bDelta}{\overline{\Delta}}

\newcommand{\tlambda}{\tilde{\lambda}}

\newcommand{\bfp}{\CP}
\newcommand{\bfP}{\mathsf{P}}

\newcommand{\fl}{\mathfrak{l}}
\newcommand{\fn}{\mathfrak{n}}

\newcommand{\te}{\tilde{e}}

\newcommand{\sub}[1]{\bar{#1}}

\newtheorem{lem}{Lemma}

\newtheorem{cor}{Corollary}
\newtheorem{rem}{Remark}
\newtheorem{thm}{Theorem}
\newtheorem{prop}{Proposition}

\begin{document}
%%%%%%%%%%%%%%%%%%%%%%%%%%%%%%%%%%%%%%%%%%%%%%%%
%%%%%%% title page %%%%%%%%%
\begin{CJK*}{UTF8}{min}

\title[Determinant for Parabolic Verma Modules of Lie Superalgebras]{Determinant Formula for Parabolic Verma Modules of Lie Superalgebras}

\author[Yoshiki Oshima]{Yoshiki Oshima (大島芳樹)}
\address{(Y.\ O.) Kavli IPMU (WPI), University of Tokyo, Kashiwa, Chiba 277-8583, Japan}
\email{yoshiki.oshima(at)ipmu.jp}

\author[Masahito Yamazaki]{Masahito Yamazaki (山崎雅人)}
\address{(M.\ Y.) Kavli IPMU (WPI), University of Tokyo, Kashiwa, Chiba 277-8583, Japan;
School of Natural Sciences, Institute for Advanced Study, Princeton NJ 08540, USA}
\email{masahito.yamazaki(at)ipmu.jp}
\urladdr{http://member.ipmu.jp/masahito.yamazaki/}

\begin{abstract}
We prove a determinant formula
for a parabolic Verma module of a Lie superalgebra, previously conjectured by the second author.
Our determinant formula generalizes the previous results of Jantzen for a parabolic Verma module of a (non-super) Lie algebra,
and of Kac concerning a (non-parabolic) Verma module for a Lie superalgebra.
The resulting formula is expected 
to have a variety of applications in the study of higher-dimensional supersymmetric conformal field theories.
We also discuss irreducibility criteria for the Verma module.
\end{abstract}
%%%%%%%%%%%%%%%%%%%%%%%%%%%%%%%%%%%%%%%%%%%%%%%%%

\maketitle 
\end{CJK*}

\tableofcontents

%%%%%%%%%%%%%%%%%%%%%%%%%%%%%%%%%%%%%%%%%%%%%%%%

\section{Introduction}\label{sec.intro}

The study of Verma modules \cite{Verma} is a rich subject in the representation theory of Lie algebras and their universal enveloping algebras (see e.g.\ \cite{Diximier,Jantzen_LNM,Humphreys} and references therein).
Given a Verma module, one natural question is to ask precisely when the Verma module is irreducible/reducible, and if reducible 
to determine the composition factors of the module.

One useful approach to this problem is to consider the determinant of the contravariant form 
on the Verma module. 
For a Verma module associated with a finite-dimensional semisimple Lie algebra, Jantzen \cite{Jantzen_LNM} and Shapovalov \cite{Shapovalov} 
proved a general formula for this determinant.
Since then this result has been generalized in several different directions.
In one direction, Jantzen \cite{Jantzen} derived a similar determinant formula 
for a parabolic generalization of a Verma module\footnote{A parabolic Verma module is also called a generalized Verma module in the literature.}. 
This formula reduces to the Shapovalov formula \cite{Shapovalov} (see also \cite{Jantzen73})
when the parabolic subalgebra is chosen to be a Borel subalgebra. 
In another direction, Kac \cite{Kac_1986}
derived a determinant formula for a Verma module of a Lie superalgebra.
The natural generalization, then, is to consider the determinant formula
for a parabolic Verma module of
a Lie superalgebra. It seems that 
such a generalization has never been discussed in the literature.
The goal of this paper is to fill in this gap.

While this might sound like a small technical problem,
there have recently been strong motivations to study this generalization, 
coming from the physics of higher-dimensional (super) conformal field theories (CFTs) (see \cite{Rychkov:2016iqz,Simmons-Duffin:2016gjk} for recent reviews).
It has recently been pointed out \cite{Penedones:2015aga,Yamazaki:2016vqi}
 that the study of parabolic Verma modules, and in particular the determinant formula,
is an important ingredient in the study of higher-dimensional CFTs, for example
in the derivation of the unitarity bound \cite{Evans,Mack:1975je,Minwalla:1997ka,Ferrara:2000nu}  and in the the derivation of the recursion relations for conformal blocks \cite{Penedones:2015aga} (see also \cite{Kos:2013tga,Kos:2014bka,Iliesiu:2015akf,Costa:2016xah}). This is perhaps not too surprising when contrasted with the study of the 2d CFTs, where the Kac determinant formula \cite{Kac_LNP,Feigin_Fuks_1982} plays a prominent role in the representation theory of the Virasoro algebra.
In applications to higher-dimensional CFTs 
it is crucial to consider parabolic Verma modules, as opposed to Verma modules for Borel subalgebras.

In this paper we present a precise mathematical formulation and proof of the previously-conjectured \cite{Yamazaki:2016vqi} determinant formula. 
We will also work out some mathematical consequences of our main theorem, and in particular
derive irreducibility criteria for the parabolic Verma module,
which will be of great use in practical applications.
We include some preliminary comments on CFT applications, while more details will be left out for future project.

This paper is organized as follows. In section \ref{sec.formulation} we present our main theorem (Theorem \ref{thm.main}),
explaining the necessary ingredients along the way. The proof of Theorem \ref{thm.main} is given in section \ref{sec.proof}. In section \ref{sec.simplicity} we derive irreducibility criteria of the parabolic Verma modules,
see Theorem \ref{prop_1}, Theorem \ref{thm.simple}, as well as Proposition \ref{prop.M}, Proposition \ref{prop.Mp}
and Corollary \ref{cor.Msharp}.

%%%%%%%%%%%%%%%%%%%%%%%%%%%%%

\subsection*{Acknowledgments}
This research is supported by WPI program (MEXT, Japan). MY is also supported by JSPS Program for Advancing Strategic International Networks to Accelerate the Circulation of Talented Researchers, by JSPS KAKENHI Grant No.\ 15K17634, and by Institute for Advanced Study. 

%%%%%%%%%%%%%%%%%%%%%%%%%%
\section{Statement of the Theorem}\label{sec.formulation}

In this section we state our main theorem (Theorem \ref{thm.main}) with a minimal set of definitions and notations.
The proof of the main theorem will be given in the next section.

In this paper $\mathbb{Z}, \mathbb{N}$ and $\mathbb{Q}$ denote the set of integers, non-negative integers and rational numbers, respectively.

%-------------------------------------------------------------------------------------------------------------------------------------
\subsection{Lie Superalgebra}

Let us recall the basics of Lie superalgebras and fix some notation.
See e.g.\ \cite{KacAdvM,Frappat:1996pb,Musson} for details.

Let  $\mathfrak{g}$ be a finite-dimensional Lie superalgebra, whose even (odd) degree part we denote by $\frakg_{\sub{0}}$ ($\frakg_{\sub{1}}$). 
Throughout the paper we assume that our superalgebras are defined over $\mathbb{C}$;
it is possible to repeat a similar argument for a more general 
algebraically closed field of characteristic zero.

In this paper we assume that $\mathfrak{g}$ is a contragredient finite-dimensional Lie superalgebra with an indecomposable Cartan matrix (see \cite{KacAdvM,Musson} for definition).
Then $\frakg$ has a natural invariant bilinear form $(-,-)$.\footnote{Here a bilinear form $(-, -)$ on $\mathfrak{g}$ is invariant if
$([a,b],c)=(a, [b,c])$ for all $a, b, c \in \mathfrak{g}$.} 
Such superalgebras have been classified,
and are either a simple Lie algebra, or one of the following superalgebras;
\beq
&A(m,n) = \mathfrak{sl}(m+1| n+1) \;, \quad m>n\ge 0 \;,\\
&A(n,n) = \mathfrak{gl}(n+1| n+1) \;, \\
&B(m,n) = \mathfrak{osp}(2m+1| 2n) \;, \quad m\ge 0\;,\ n>0 \;, \\
&C(n) = \mathfrak{osp}(2| 2n-2) \;, \quad n\ge 2 \;, \\
&D(m,n) = \mathfrak{osp}(2m| 2n) \;, \quad m\ge2\;,\ n\ge 1 \;, \\
&D(2,1;\alpha) \quad \alpha \ne 0, 1\;, \quad
G(2)  \;, \quad F(4) \;. \\
\eeq
These superalgebras are simple except for $\mathfrak{gl}(n+1|n+1)$.  In the case of type $A(n,n)$ the Cartan matrix is degenerate and the corresponding simple Lie superalgebra is $\mathfrak{psl}(n+1|n+1)$.

As we will comment later in Remark \ref{rem.CFT},
all the superalgebras
corresponding to superconformal algebras in 
dimensions greater than two either appear in the above list or 
is $\mathfrak{psl}(n+1|n+1)$.

Let $\frakh$ be a Cartan subalgebra of $\frakg_{\sub{0}}$.
In our setting it is known that $\frakh$ %=\frakh_0$ 
is self-normalizing
 (see \cite[Section 8.3.1]{Musson}).
We have the root space decomposition 
\beq
\frakg = \frakh \oplus \bigoplus_{\alpha\in \frakh^*\setminus\{0\}}
 \frakg^{\alpha} \;,
\eeq
where $ \frakg^{\alpha}$ is the root space corresponding to $\alpha$, i.e.\
\beq
\frakg^{\alpha} := \{x\in \frakg \, |\,  [h,x] =  \alpha(h) (x) \textrm{ for all } h\in \frakh \} \;.
\eeq

%-------------------------------------------------------------------------------------------------------------------------------------
\subsection{Root System and Parabolic Subalgebra}

Let us define the set of roots to be 
\begin{align*}
\Delta := \{ \alpha\in\frakh^* | \, \alpha\ne 0 \textrm{ and } \frakg^{\alpha}\ne 0 \}\;.
\end{align*}
Let us denote the set of even and odd roots by $\Delta_{\sub{0}}$
 and $\Delta_{\sub{1}}$, respectively:
\beq
\Delta= \Delta_{\sub{0}} \cup \Delta_{\sub{1}}\;.
\eeq
It is known that our assumption implies $\textrm{dim}\, \frakg^{\alpha}= 1$ for all roots $\alpha\in\Delta$.

Let us next choose a set of positive roots. An element $ h \in \frakh$ is called 
{\it regular} if the real part $\operatorname{Re}(\alpha(h))$ of $\alpha(h)$ is 
non-zero for all roots $\alpha$.
Any regular element $h\in \frakh$ determines a decomposition
\beq
\Delta = \Delta^{+} \cup \Delta^{-}\;,
\label{delta_pm}
\eeq
where 
\beq
\Delta^{\pm} := \{
 \alpha \in \Delta | \pm \operatorname{Re}(\alpha(h)) >0 
\} \;.
\eeq
Note that the decomposition \eqref{delta_pm}
depends on the choice of $h$. Similarly, the Borel subalgebra $\mathfrak{b}$
(as defined by \cite{PenkovSerganova94})
\beq
\mathfrak{b}:= \frakh\oplus \bigoplus_{\alpha\in \Delta^{+}} \frakg^{\alpha} 
\eeq
depends in general on the choice of $h$ (even up to automorphisms of $\frakg$).

Let $\Pi$ be the set of simple roots in $\Delta^{+}$, namely, $\Pi$ consists of the positive roots $\alpha\in \Delta^{+}$ which cannot be written as a sum of two positive roots.
The set of positive roots is decomposed into that of positive even roots 
$\Delta_{\sub{0}}^+$ and that of positive odd roots 
$\Delta_{\sub{1}}^+$: $\Delta^+=\Delta_{\sub{0}}^+ \cup \Delta_{\sub{1}}^+$. 
Let us define the root lattice $Q(\Delta)$ and its positive part $Q^+(\Delta)$ by
\beq
Q(\Delta):=\mathbb{Z} \Delta \subset \mathfrak{h}^* \;, \quad
Q^+(\Delta):=\mathbb{N} \Delta^+ \;.
\eeq
In the set of weights $\frakh^*$ we define an ordering by 
\beq
\lambda\geq \mu\quad \Longleftrightarrow\quad \lambda-\mu\in Q^+(\Delta) \;.
\eeq
Let us define the set $\overline{\Delta}_{\sub{0}}$ and 
the set of {\it isotropic} roots $\overline{\Delta}_{\sub{1}}$ by
\beq
\overline{\Delta}_{\sub{0}} :=\{
\alpha \in \Delta_{\sub{0}} |\, \alpha/2 \not \in \Delta_{\sub{1}}
\}\subset \Delta_{\sub{0}} \;,\quad
\overline{\Delta}_{\sub{1}} :=\{
\alpha \in \Delta_{\sub{1}} |\, 2\alpha \not \in \Delta_{\sub{0}}
\} \subset \Delta_{\sub{1}}\;.
\eeq
Also, write
\beq
\overline{\Delta}^+_{\sub{0}}=\overline{\Delta}_{\sub{0}}\cap \Delta^+\;,\quad
\overline{\Delta}^+_{\sub{1}}=\overline{\Delta}_{\sub{1}}\cap \Delta^+\;.
\eeq
It is known that an odd root $\alpha$ is isotopic
 (i.e.\ $\alpha\in \overline{\Delta}_{\sub{1}}$)
 if and only if $(\alpha, \alpha)=0$.
The set of non-isotropic roots is then defined to be
 the complement of $\bDelta_{\sub{1}}$:
\beq
\Delta_{\rm non\mathchar`-iso}:=
 \Delta_{\sub{0}} \, \cup \, 
 (\Delta_{\sub{1}} \setminus \overline{\Delta}_{\sub{1}}) \;.
\eeq

For a non-isotropic root $\alpha\in \Delta$ let us denote 
its Weyl reflection by $s_{\alpha}$ and define
 its coroot $\alpha^\vee\in\frakh$ by
\beq
\langle \lambda,\alpha^\vee\rangle=
 \frac{2(\lambda,\alpha)}{(\alpha, \alpha)} 
\eeq
for all $\lambda\in \frakh^*$. Here $\langle -, -\rangle$ denotes the natural pairing between $\frakh$ and $\frakh^*$, and $(-,-)$ is the bilinear form on
 $\frakh^*$ induced from that on $\frakh$ mentioned previously.
We also define $h_{\alpha}\in \frakh$ for $\alpha\in\frakh^*$ by 
\beqn\label{halpha_def}
\langle \lambda,h_{\alpha} \rangle=
(\lambda,\alpha)
\eeqn
for $\lambda\in \frakh^*$. 
Note that $[\frakg^{\alpha},\frakg^{-\alpha}]=\bC h_{\alpha}$.

We denote the Weyl vector by $\rho$:
\beq
\rho:=\frac{1}{2} \left(
\sum_{\alpha \in \Delta^+_{\sub{0}}} \alpha  
-
\sum_{\alpha \in \Delta^+_{\sub{1}} }\alpha  
\right) \ .
\eeq

%-------------------------------------------------------------------------------------------------------------------------------------
\subsection{Parabolic Verma Modules}

Let us choose a subset $\Pi_{\fl}\subset \Pi\cap\Delta_{\sub{0}}$.
We define 
\begin{align*}
&\Delta_{\fl}:= \bZ \Pi_{\fl} \cap \Delta \;, \quad  
Q(\Delta_{\fl}):= \bZ \Delta_{\fl} \;, \quad
\Delta_{\fl}^{+}:=\Delta^{+} \cap \Delta_{\fl} \;,  \quad 
Q^+(\Delta_{\fl}):= \bN \Delta_{\fl}^+ \;,  \\
&\Delta_{\fn}:=\Delta^+\setminus \Delta_{\fl} \;, \quad
\Delta_{\fn,\sub{0}}:=\Delta_{\fn}\cap\Delta_{\sub{0}}\;, \quad
\overline{\Delta}_{\fn,{\sub{0}}}:=\Delta_{\fn}\cap\overline{\Delta}_{\sub{0}}.
\end{align*}

Let us define the subalgebras $\frakp, \fl, \fn$ by
\begin{align*}
\begin{split}
&\fl : = \frakh\oplus \bigoplus_{\alpha\in \Delta_{\fl}} \frakg^{\alpha} \;,
 \qquad
\fn :=  \bigoplus_{\alpha\in \Delta_{\fn}} \frakg^{\alpha}\;, \\
&\mathfrak{p}
:= \frakh\oplus \bigoplus_{\alpha\in \Delta_{\fl} \cup \Delta^+} \frakg^{\alpha}
= \fl \oplus \fn \;.
\end{split}
\end{align*}
It then follows that $\fl\subset\frakg_{\sub{0}}$, and the subalgebra 
$\fl$ is a reductive Lie algebra with root system $\Delta_{\fl}$.
Let us denote the Weyl group of $\fl$ 
by $W_{\fl}$.
We note that $\Delta_{\fn}=\Delta_{\fn,\sub{0}}\sqcup\Delta^+_{\sub{1}}$.
The subalgebra $\mathfrak{p} (\supset \mathfrak{b}) $ is a parabolic subalgebra of $\mathfrak{g}$,
and we consider parabolic Verma modules with respect to this subalgebra.
Note that when $\Pi_{\fl}=\varnothing$ then $\frakp$ coincides with the Borel subalgebra $\mathfrak{b}$. 

Define the sets of weights
\begin{align*}
&P(\Delta_{\fl}) := \{
 \lambda \in \frakh^* | \, \langle \lambda, \alpha^\vee\rangle \in \mathbb{Z} \, \,  (\forall \alpha \in \Pi_{\fl})
\} \;, \\ 
&P^+(\Delta_{\fl}) := \{
 \lambda \in \frakh^* | \, \langle \lambda, \alpha^\vee\rangle \in \mathbb{N} \, \,  (\forall \alpha \in \Pi_{\fl})
\} \;.
\end{align*}
For a given weight $\lambda \in P^+(\Delta_{\fl})$ consider 
a finite-dimensional irreducible representation $V(\lambda)$ of $\fl$ with
highest weight $\lambda$.
We can regard this as a representation of the parabolic subalgebra $\mathfrak{p}$, 
by letting $\mathfrak{n}$ act trivially on $V(\lambda)$.
The representation $V(\lambda)$ here is then naturally regarded as a representation of the universal enveloping algebra 
$\frakU(\mathfrak{p})$.
We are interested in the parabolic Verma module
\beq
M_{\mathfrak{p}}(\lambda):= \textrm{Ind}_{\frakp}^{\frakg} (V(\lambda))=\frakU(\mathfrak{g}) \otimes_{\frakU(\mathfrak{p}) } V({\lambda}) \;.
\eeq
This module has a weight space decomposition:
\beq
M_{\mathfrak{p}}(\lambda)=
\displaystyle\bigoplus_{\mu \leq \lambda} M_{\mathfrak{p}}(\lambda)^{\mu} \;.
\eeq 
In the case $\Pi_{\fl}=\varnothing$ (with $\lambda\in \frakh^*$ arbitrary) 
we denote the parabolic Verma module $M_{\mathfrak{p}}(\lambda)$ by $M(\lambda)$. This
$M(\lambda)$ is called a Verma module, or more explicitly a non-parabolic Verma module.
We embed $V(\lambda)$ into $M_{\mathfrak{p}}(\lambda)$
as $1\otimes V(\lambda)$.
Then 
\beq
V(\lambda)= \bigoplus_{\nu\in \lambda-Q^+(\Delta_{\fl})} M_{\mathfrak{p}}(\lambda)^{\nu}\;.
\eeq

%-------------------------------------------------------------------------------------------------------------------------------------
\subsection{Characters}

Let $\bZ^{\frakh^*}$ be the additive group of all maps $\frakh^*\to \bZ$.
An element $\chi\in \bZ^{\frakh^*}$ is written as $\chi=\sum_{\lambda\in\frakh^*}\chi_\lambda e^{\lambda}$. Then the support of the map $\chi$ is defined as $\operatorname{Supp}(\chi)=\{\lambda\in\frakh^* : \chi_\lambda\neq 0\}$.
Let $\bZ\langle \frakh^* \rangle$ be the subset of $\bZ^{\frakh^*}$ consisting of $\chi$ whose support is contained in a set $\bigcup_{i=1}^n \{\mu|\, \mu\leq \lambda_i\}$ for some finitely many weights $\lambda_1,\dots,\lambda_n\in\frakh^*$.
We can define the multiplication in $\bZ\langle \frakh^* \rangle$ by extending
 the rule $e^{\lambda}\cdot e^{\mu} = e^{\lambda+\mu}$, namely,
 we define the product $\chi\cdot\chi'$ for $\chi,\chi'\in\bZ\langle \frakh^* \rangle$ by 
$(\chi\cdot\chi')_\nu=\sum_{\mu\in\frakh^*}\chi_{\nu-\mu}\chi'_{\mu}$.
Then $\bZ\langle \frakh^* \rangle$ becomes an algebra with the unit $e^0=1$.
We can also define an infinite sum $\sum_i \chi^i$
 of $\chi^i\in \bZ\langle \frakh^* \rangle$
 if for any $\mu\in \frakh^*$
 we have $(\chi^i)_{\mu}=0$ for all but finitely many $i$.

For a weight module $M=\bigoplus_{\lambda\in\frakh^*}M^{\lambda}$ we define its character as ${\rm ch}\, M:=\sum_{\lambda\in\frakh^*}(\dim M^{\lambda})e^{\lambda}$.  In what follows we will consider modules belonging to the category $\mathcal{O}$ \cite{Bernshtein}, and in such cases we have ${\rm ch}\, M\in\bZ\langle\frakh^*\rangle$.

For $\lambda\in\frakh^*$
 a partition of $\lambda$ with respect to $\Delta^+$
 is a map $\pi:\Delta^+\to \bN$ such that $\pi(\alpha)=0 \textrm{ or } 1$
 for $\alpha\in \Delta_{\sub{1}}^+$
 and $\sum_{\alpha\in\Delta^+}\pi(\alpha)\alpha=\lambda$.
We define $\CP(\lambda)\in \bN$ to be the number of partitions of $\lambda$.
The function $\CP$ is called the partition function with respect to $\Delta^+$.
Let us further define 
\begin{align*}
\mathfrak{P} := \sum_{\lambda\in\frakh^*} \CP(\lambda) \, e^{-\lambda}\in \bZ\langle\frakh^*\rangle \;.
\end{align*}
Then it is easy to see that 
\begin{align*}
\mathfrak{P}=\frac{\prod_{\alpha\in \Delta_{\sub{1}}^+}
 (1+e^{-\alpha})}{\prod_{\alpha\in \Delta_{\sub{0}}^+}(1-e^{-\alpha})} \;.
\end{align*}
Here, the right hand side is well-defined as an element of $\bZ\langle\frakh^*\rangle$, since
$1-e^{-\alpha}$ has an inverse element $\sum_{n=0}^{\infty} e^{-n\alpha} \in \bZ\langle\frakh^*\rangle$.
Similarly, we define $\CP_{\fl}$ (resp.\ $\CP_{\fn}$) to be the partition function with respect to $\Delta_{\fl}^+$ (resp.\ $\Delta_{\fn}$),
and correspondingly define $\mathfrak{P}_{\fl}$
 (resp.\ $\mathfrak{P}_{\fn}$) $\in \bZ\langle\frakh^*\rangle$.

For $\lambda\in \frakh^{*}$ the Poincar\`{e}-Birkhoff-Witt theorem gives
\begin{align}\label{M_gen}
{\rm ch}\, M(\lambda)=\mathfrak{P}\, e^\lambda\;.
\end{align}
The characters ${\rm ch}\, M(\lambda)$ for different $\lambda$'s are linearly independent
 in the following sense:
 %----------------------------------------------------------------------------------------------------
\begin{lem}\label{lem.cancel}
Let $\lambda\in\frakh^*$.
If 
\beq
\sum_{\mu\leq \lambda} c_{\mu}\, {\rm ch}\, M(\mu)=0  
\eeq
for some $c_{\mu}\in \bZ$,
then $c_{\mu}=0$ for all $\mu$.
\end{lem}
%----------------------------------------------------------------------------------------------------

\begin{proof}
If $c_{\mu}\neq 0$ for some $\mu$,
 we can find a maximal weight $\mu_0$ among such $\mu$'s. 
Then the coefficient of $e^{\mu_0}$ in 
 $\sum_{\mu} c_{\mu}\, {\rm ch}\, M(\mu)$
 is $c_{\mu_0}\neq 0$, which is a contradiction.
\end{proof}

%----------------------------------------------------------------------------------------------------

The character of a parabolic Verma module is given by the following lemma: 
%----------------------------------------------------------------------------------------------------------------
\begin{lem}\cite[Lemma 1]{Jantzen}\label{lem.alternate}
For all $\lambda\in P^+(\Delta_{\fl})$ we have
\beq
{\rm ch}\, M_{\mathfrak{p}}(\lambda) = \sum_{w\in W_{\fl}} {\rm det}(w) \, {\rm ch} \, M(w.\lambda) \;.
\eeq
\end{lem}

%----------------------------------------------------------------------------------------------------------------
\begin{proof}
This lemma follows from a combination of 
\beq
\textrm{ch}\, M_{\mathfrak{p}}(\lambda) =  \mathfrak{P}_{\fn} \, \textrm{ch}\, V(\lambda) \;,
\eeq
the Weyl character formula
\begin{align}\label{eq.Weyl_char}
\textrm{ch}\, V(\lambda) = \mathfrak{P_{\fl}} \sum_{w\in W_{\fl}} (-1)^w e^{w.\lambda} \;,
\end{align}
the factorization
\begin{align}\label{tripleP}
\mathfrak{P}=\mathfrak{P}_{\fl} \, \mathfrak{P}_{\fn} \;,
\end{align}
and \eqref{M_gen}.
\end{proof}
%----------------------------------------------------------------------------------------------------------------

We define for $\lambda\in P(\Delta_{\fl})$
\beq
\chi^{\mathfrak{p}}(\lambda):= \sum_{w\in W_{\fl}} \textrm{det}(w)\, \textrm{ch}\, M(w. \lambda)\;, 
\eeq
where the dot action $w.\lambda$ is defined to be
the Weyl group action, with shift by the Weyl vector $\rho$:
\beq
w.\lambda:=w(\lambda+\rho)-\rho \;.
\eeq
We can easily see that 
\beq
w.w'.\lambda=(w w').\lambda \;.
\eeq
for $w, w'\in W_{\fl}$.
For $\lambda\in P^{+}(\Delta_{\fl})$
 we have $\chi^{\mathfrak{p}}(\lambda)={\rm ch}\,M_{\frakp}(\lambda)$ 
 by Lemma~\ref{lem.alternate}.
It is easy to show from definition that $\chi^{\mathfrak{p}}(w.\lambda)=\det(w)\chi^{\mathfrak{p}}(\lambda)$. In particular, $\chi^{\mathfrak{p}}(\lambda)=0$ if $w.\lambda=\lambda$ for some $w\in W_\fl$.

For an odd isotropic root $\alpha$, let us define a 
similar expression, but with contributions from the root $\alpha$ removed:
\beq
\textrm{ch}\, M_{\alpha}(\lambda) :=\textrm{ch}\, M(\lambda)\,/ (1+e^{-\alpha}) \;.
\eeq
We then define another formal character $\chi_{\alpha}^{\frakp}$ by
\beqn
\label{chi_def}
\chi^{\mathfrak{p}}_{\alpha}(\lambda):= \sum_{w\in W_{\fl}} \textrm{det}(w)\, \textrm{ch}\, M_{w\alpha}(w.\lambda) 
= \sum_{n=0}^{\infty} (-1)^n \chi^{\frakp}(\lambda-n\alpha)
\;.
\eeqn

%----------------------------------------------------------------------------------------------------------------
\subsection{Contravariant Form}

By construction of a contragredient Lie superalgebra we can define
 an anti-automorphism $\sigma$ of $\frakg$ which
maps the root space $\frakg^{\alpha}$ to $\frakg^{-\alpha}$ and is the identity on the Cartan subalgebra $\frakh$. 
For each positive root $\alpha\in \Delta^+$ take a root vector $x_{\alpha}(\neq 0)\in\frakg^{\alpha}$ and denote its image under the anti-automorphism by $x_{-\alpha}:=\sigma(x_\alpha)\in\frakg^{-\alpha}$. We choose the normalization of $x_{\alpha}$ and $x_{-\alpha}$ so that we have $h_{\alpha}=[x_{\alpha},x_{-\alpha}]$,
where $h_{\alpha}$ was defined previously in \eqref{halpha_def}.
The anti-automorphism $\sigma$ is naturally extended to the whole of 
$\frakU(\frakg)$, which we denote by the same symbol $\sigma$.

For all $\lambda\in \mathfrak{h}^*$ there exists in the parabolic Verma module 
$M_{\mathfrak{p}}(\lambda)$
a non-zero symmetric bilinear form $(-, -)$, satisfying
\beq
(xu, v)=(u, \sigma(x)v)\textrm{ for all }x\in \frakU(\frakg) \textrm{ and } u, v \in M_{\mathfrak{p}}(\lambda) \;.
\eeq
Such a bilinear form is called 
the {\it contravariant form} in the literature.

For a Verma module $M(\lambda)$ (i.e.\ when $\frakp=\mathfrak{b}$), such a bilinear form can be 
explicitly constructed by the Harish-Chandra projection \cite[Section 8.2]{Musson}.
The surjection $M(\lambda) \twoheadrightarrow M_{\frakp}(\lambda)$ then induces a 
contravariant form on a general parabolic Verma module $M_{\frakp}(\lambda)$.
The weight spaces for different weights are orthogonal 
with respect to this bilinear form.

\begin{rem}

We use the same symbol $(-, -)$ to denote both the bilinear form on the Lie superalgebra $\frakg$
and the contravariant form on the parabolic Verma module $M_{\frakp}(\lambda)$. We hope context makes it clear 
which is meant by this notation.

\end{rem}

%----------------------------------------------------------------------------------------------------------------
\subsection{Determinant Formula}

Let us choose $v\in M_{\mathfrak{p}}(\lambda)^{\lambda}$ with $v\ne 0$,
 which we normalize to be $(v,v)=1$.
For a weight $\mu\leq \lambda$ we define $D(\lambda; \mu)$
 as the determinant of $(-,-)$ in $M_{\mathfrak{p}}(\lambda)^{\mu}$
 with respect to the basis given below.

%----------------------------------------------------------------------------------------------------------------

For all $\nu\in \lambda-Q^+(\Delta_{\fl})$, take an orthonormal basis
 $(e_{\nu, i})_{1\le i\le n(\nu)}$ of $V(\lambda)^{\nu}
 =M_{\mathfrak{p}}(\lambda)^{\nu}$, with $n(\nu)$ denoting the 
 dimension of $V(\lambda)^{\nu}$. 
Then $(e_{\nu, i})_{\nu,\, 1\le i\le n(\nu)}$ form a basis of $V(\lambda)$.

Let $\bfP(\eta)$ be the set of partitions of $\eta$
 with respect to $\Delta_{\fn}$, namely, $\bfP(\eta)$ is the set of maps
$\pi: \Delta_{\fn}\to \mathbb{N}$ such that 
$\pi(\alpha)=0 \textrm{ or } 1$
 for $\alpha\in \Delta_{\fn}\cap \Delta_{\sub{1}}$
and 
$\eta= \sum_{\alpha\in \Delta_{\fn}} \pi(\alpha)\, \alpha$.
For $\pi: \Delta_{\fn}\to \mathbb{N}$, we set
$S(\pi):= \sum_{\alpha\in\Delta_{\fn}} \pi(\alpha)\alpha$
 and $|\pi|:=\sum_{\alpha\in\Delta_{\fn}} \pi(\alpha)$.

It follows from the Poincar\`{e}-Birkhoff-Witt theorem that the parabolic Verma module $M_{\mathfrak{p}}(\lambda)$
is spanned by $x_{-\pi} e_{\nu, i}$,
where we defined
\beq
x_{-\pi}:=\prod_{\alpha\in\Delta_{\fn}} x_{-\alpha}^{\pi(\alpha)} \;.
\eeq
Here the product is taken in the order determined from a fixed ordering of $\Delta_{\fn}$. The determinant formula below turns out to be independent of this 
ordering choice.

Having fixed a basis, the determinant $D(\lambda; \mu)$
can now be defined as the determinant of the matrix
 $(x_{-\pi} e_{\nu, i}, x_{-\pi'} e_{\nu', j})_{(\pi,\nu,i),(\pi',\nu',j)}$,
where the indices run over 
\begin{align*}
\nu, \nu' \in \lambda-Q^+(\Delta_{\fl}), \quad
1\le i\le n(\nu),\ 1\le j\le n(\nu'), \quad
\pi\in \bfP(\nu-\mu),\ \pi'\in \bfP(\nu'-\mu).
\end{align*}
It is easy to see that the determinant $D(\lambda; \mu)$ does not depend on the choice of orthonormal basis $(e_{\nu, i})$.

Our main result is the formula for this determinant $D(\lambda;\mu) $.

%----------------------------------------------------------------------------------------------------------------

\begin{thm}[Determinant Formula]\label{thm.main}

For $\lambda\in P^+(\Delta_{\fl})$ and $\mu\leq \lambda$
 the determinant  $D(\lambda;\mu)$ is given by
\begin{empheq}[box={\mybluebox[7pt]}]{align}
\begin{split}
&D(\lambda; \mu)  =  c\, D_1\, D_2\, D_3  \;, \\
&
D_1:= \prod_{\alpha\in {\bDelta_{\fn, \sub{0}}}} \prod_{r=1}^{\infty} 
\left(
( \lambda+\rho , \alpha )- \frac{r}{2} (\alpha, \alpha)
\right)^{\chi^{\mathfrak{p}}(\lambda-r \alpha)_{\mu}} \;,\\
&D_2:= \prod_{\alpha \in \Delta_{\sub{1}}^+
 \setminus \overline{\Delta}_{\sub{1}}^+} 
\prod_{r=1}^{\infty}
\left(
( \lambda+\rho , \alpha )
-\frac{2r-1}{2}(\alpha, \alpha) \right)^{\chi^{\mathfrak{p}}(\lambda-(2r-1)\alpha)_{\mu}} \;, \\
&D_3:= \prod_{\alpha \in \overline{\Delta}_{\sub{1}}^+} 
\left( \lambda+\rho , \alpha 
\right)^{\chi^{\mathfrak{p}}_{\alpha}(\lambda-\alpha)_{\mu}}  \;.
\label{fermionic_determinant}
\end{split}
\end{empheq}
Here, $c$ is a non-zero constant which depends only on $\lambda-\mu$.
\end{thm}

%----------------------------------------------------------------------------------------------------------------

\begin{rem}

The first product is over
 $\bDelta_{\fn,\sub{0}}:=\Delta_{\fn}\cap\overline{\Delta}_{\sub{0}}$
and not over the whole $\Delta_{\fn,\sub{0}}$.
Also, the power in the expression of $D_3$ is $\chi^{\frakp}_{\alpha}$,
 not $\chi^{\mathfrak{p}}$.

\end{rem}

\begin{rem}
The integers in the exponents, 
$\chi^{\frakp}(\lambda-s\alpha)_\mu$
and $\chi^{\frakp}_\alpha(\lambda-\alpha)_\mu$,
can be negative in general.
However, it will turn out that $D_1D_2D_3$
 as a function on $\lambda\in P^+(\Delta_{\fl})$ with $\lambda-\mu$
 fixed is a polynomial in $\lambda$
 on each connected component of $P^+(\Delta_{\fl})$.
The product $D_1D_2D_3$ is thus well-defined for all
 $\lambda\in P^+(\Delta_{\fl})$.
\end{rem}

%----------------------------------------------------------------------------------------------------------------
\begin{rem}

Theorem \ref{thm.main} was conjectured recently in \cite{Yamazaki:2016vqi}.
When $\Pi_{\fl}=\varnothing$, $\mathfrak{p}$ is then a Borel subalgebra, and
Theorem \ref{thm.main} is reduced to the result by Kac \cite{Kac_1986}.\footnote{The original proposal of \cite{Kac_LNM,Kac_LNP} contained an error, which was later corrected in \cite{Kac_1986}.}
For the non-super case this is reduced to the result of Jantzen \cite[Satz2]{Jantzen}.
For the non-super and Borel case, this is reduced to the Shapovalov determinant formula \cite{Shapovalov}.

\end{rem}

%----------------------------------------------------------------------------------------------------------------

\begin{rem}
For our proof of Theorem \ref{thm.main} the assumption that
 $\fl\subset\frakg_{\sub{0}}$ is crucial.
It would be interesting to generalize our theorem to the case
 $\fl \not\subset \frakg_{\sub{0}}$.
\end{rem}

%----------------------------------------------------------------------------------------------------------------

\begin{rem}\label{rem.CFT}
As commented in introduction, the relevance of the determinant formula
for higher-dimensional CFTs has recently been discussed in \cite{Penedones:2015aga,Yamazaki:2016vqi}.
In this application, we choose $\mathfrak{g}$ to be a superconformal algebra,
which was classified by Nahm \cite{Nahm:1977tg}.
In all the cases, the even part $\frakg_{\sub{0}}$ takes the form
\beq
\frakg_{\sub{0}} = \mathfrak{so}(D, 2) \oplus \mathfrak{g}_{\rm R} \;,
\eeq
where $D$ is the spacetime dimension and the subalgebra $\mathfrak{g}_{\rm R}$ represents the so-called R-symmetry
of the superconformal algebra.
The subalgebra $\frakg'$ is then taken to be
\beq
\fl=\mathfrak{so}(D)\oplus \mathfrak{so}(2) \oplus \mathfrak{g}_{\rm R} \;.
\eeq
where $\mathfrak{so}(D)$ is the subgroup of spatial rotations and $\mathfrak{so}(2)$ is generated by the dilatation operator.
In particular $\Pi_{\fl}$ contains no odd roots.
The subalgebra $\fn$ is Abelian, and is generated by the so-called special conformal boosts.
The element $\lambda\in P^+(\Delta_{\fl})$ is then given by a pair of the conformal dimension
and the spins under the rotation group, which together specify a conformal primary operator.

\end{rem}

The next section will be devoted to the proof of Theorem \ref{thm.main}.

%%%%%%%%%%%%%%%%%%%%%%%%%%%%%%%%%%%%%%%%%%%%%%%%

\section{Proof of Theorem \ref{thm.main}}\label{sec.proof}

%%%%%%%%%%%%%%%%%%%%%%%%%%%%%%%%%%%%%%%%%%%%%%%%
\subsection{Outline of the Proof}

Before coming to the detailed proof, 
let us first sketch the outline.
Readers not interested in such an overview can safely skip this subsection.

Our proof consists of three steps.

The first step is to determine the leading term of the determinant $D(\lambda; \mu)$.  We view $D(\lambda; \mu)$ as a polynomial in the highest weight $\lambda$.  This is carried out by changing the coefficient ring from  $\mathbb{C}$
to the polynomial ring $A=\mathbb{C}[\{T_{\alpha}\}_{\alpha \in \Pi\setminus \Pi_{\fl}}]$, with indeterminates $T_{\alpha} (\alpha\in \Pi\setminus \Pi_{\fl})$
 and highest weight $\tlambda$ as in \eqref{def_tlambda}.
We can repeat the definitions of the previous section
 in this coefficient ring, and then
 the determinant, which we denote by $D(\lambda; \mu)_A$,
 is a polynomial in the indeterminates $T_{\alpha}$.
We can then define its leading term as the top degree part
 with respect to the total degree for $T_{\alpha}$.
It turns out that only the diagonal entries of the matrix contributes to this leading term,
and hence we easily obtain the leading term. This gives Proposition \ref{prop.leading}.

The next step is to locate possible positions of the zeros of the determinant.
This can be substantially constrained by a 
necessarily condition for the existence of zeros (Proposition \ref{prop_zero}),
leading to a conclusion that the determinant should be (up to a non-zero constant) a product of 
degree-one polynomials \eqref{eq.prop_zero} corresponding to quasi-roots.

The final step of the argument is to determine the multiplicity of each hyperplane factor inside the determinant.
Again, the trick of changing the coefficient ring to the polynomial ring $A$ in \eqref{def_A} helps here.
The basic idea is that 
to determine the multiplicity all we need to do is to count the order of the zeros when we perturb
the highest weight by an infinitesimal parameter. The variables $T_{\alpha}$ exactly do this job.
The Jantzen filtration \cite{Jantzen,Jantzen_LNM} is a powerful tool for this computation.

Overall, our proof relies heavily on the proof of \cite{Jantzen}.
We in particular reproduce several arguments from \cite{Jantzen} for self-containedness and the convenience of the reader.
It should be pointed out, however, that our proof differs from that of \cite{Jantzen}
in a number of key aspects, in particular in the considerations of 
isotropic odd roots.
Our proof also uses ideas of \cite[section 10]{Musson},
 which gives an explicit proof of \cite{Kac_1986}.

%-------------------------------------------------------------------------------------------------------------------------------------
\subsection{Coefficient Change}

Let $A$ be a $\bC$-algebra.
We can define a parabolic Verma module and hence its determinant in the coefficient ring $A$.
To do this we need to change the coefficient for many of the ingredients we have discussed so far.
We do not bother to repeat all the definitions/results in $A$-coefficient, since the discussion is completely parallel.
For example, the universal enveloping algebra is now given by $\frakU(\frakg)_A=\frakU(\frakg) \otimes_{\mathbb{C}} A$;
similarly for $\frakU(\frakp)_A$. 
Also $P^+(\Delta_{\fl})_A$ is defined as
\beq
P^+(\Delta_{\fl})_A
=\{\lambda\in
 \frakh^*_A:=\frakh^*\otimes_{\bC} A | \, \langle \lambda, \alpha^\vee\rangle \in \mathbb{N} \, \,  (\forall \alpha \in \Pi_{\fl})\} \;.
\eeq
Then the parabolic Verma module, now with $A$-coefficient,
is defined as 
\beq
M_{\mathfrak{p}}(\lambda)_A :=\frakU(\frakg)_A \otimes_{\frakU(\mathfrak{p})_A} V(\lambda)_A
\eeq
for $\lambda\in P^+(\Delta_{\fl})_A$.
In the case $\Pi_{\fl}=\varnothing$ we denote $M_{\mathfrak{p}}(\lambda)_A$ by $M(\lambda)_A$.

Suppose that there is a ring homomorphism $\varphi: A\to A'$.
This canonically induces corresponding morphisms in many of the ingredients we have,
and for simplicity of notation we denote all these morphisms by $\varphi$.
For example, we have induced morphisms $P^+(\Delta_{\fl})_A \to P^+(\Delta_{\fl})_{A'}$,
$V(\lambda)_A\to V(\varphi(\lambda))_{A'}$ and 
$M_{\mathfrak{p}}(\lambda)_A\to M_{\mathfrak{p}}(\varphi(\lambda))_{A'}$,
which are compatible with $\mathfrak{U}(\frakg)_A\to \frakU(\frakg)_{A'}$. All these morphisms are denoted by $\varphi$.
For characters we have $\varphi({\rm ch}\, M_{\mathfrak{p}}(\lambda)_A)={\rm ch}\, M_{\frakp}(\varphi(\lambda))_{A'}$,
and for the determinant we have $\varphi(D(\lambda; \mu)_A) = D(\lambda; \varphi(\mu))_{A'}$.

Following the strategy outlined above, let us first consider a change of the coefficient ring
from $\mathbb{C}$ into a polynomial ring
\beqn
A:=\mathbb{C}[\{T_{\alpha}\}_{\alpha \in \Pi\setminus \Pi_{\fl}}] \;,
\label{def_A}
\eeqn
where $\{T_{\alpha}\}_{\alpha \in \Pi\setminus \Pi_{\fl}}$ are algebraically independent over $\mathbb{C}$.

For a root $\alpha\in \Pi$ we define a weight $\omega_{\alpha}\in\mathfrak{h}^*_A$ by:
\begin{align*}
&\langle\omega_{\alpha},\alpha^\vee \rangle=1
 \text{ if $\alpha\in \Pi_{\fl}$} \;,
\qquad
\langle\omega_{\alpha},h_\alpha  \rangle=1
 \text{ if $\alpha\in \Pi\setminus \Pi_{\fl}$} \;,\\
&\langle\omega_{\alpha},h_\beta  \rangle=0
 \text{ if $\alpha, \beta \in \Pi$ and $\alpha\neq \beta$} \;.
\end{align*}
Let us define a weight $\tlambda\in P^+(\Delta_{\fl})_A$ by 
\begin{align}
\tlambda+\rho = \sum_{\alpha \in \Pi_{\fl}} r_{\alpha} \omega_{\alpha}+ \sum_{\alpha \in \Pi\setminus \Pi_{\fl}} T_{\alpha} \omega_{\alpha} \,
\label{def_tlambda}
\end{align}
with $r_{\alpha}\in \mathbb{N}$.
A basis of the parabolic Verma module $M_{\frakp}(\tlambda)_A$
 is now given by $x_{-\pi} \tilde{e}_{\nu, i}$,
 where $(\tilde{e}_{\nu, i})_{1\leq i\leq n(\nu)}$
 is an orthonormal basis of $M_{\frakp}(\tlambda)^{\nu}_A$
 for $\nu\in \tlambda-Q^{+}(\Delta_{\fl})$.
We consider the determinant $D(\tlambda;\mu)_A$ with respect to this basis.

By a specialization homomorphism 
\begin{align*}
\varphi: A=\mathbb{C}[\{T_{\alpha}\}_{\alpha \in \Pi\setminus \Pi_{\fl}}]  \twoheadrightarrow \mathbb{C} \;,
\end{align*}
we can go back to the $\mathbb{C}$-coefficient.
If $\lambda\in\frakh^*(=\frakh^*_\bC)$ and $\langle\lambda,\alpha^\vee\rangle=r_{\alpha}$ for $\alpha\in \Pi_{\fl}$, we can take $\varphi(T_{\alpha})=\langle\lambda,h_\alpha\rangle$ for $\alpha\in \Pi\setminus\Pi_{\fl}$
 so that $\varphi(\tlambda)=\lambda$.

%-------------------------------------------------------------------------------------------------------------------------------------
\subsection{Leading Term}

In the coefficient ring $A$ the determinant $D(\tlambda; \mu)_A$
is a polynomial in $\{T_{\alpha}\}_{\alpha\in \Pi\setminus \Pi_{\fl}}$. Let us discuss the leading term,
where the degree here refers to the total degree of all $\{T_{\alpha}\}_{\alpha\in \Pi\setminus \Pi_{\fl}}$ (i.e.\ it is the degree
when we collapse $T_{\alpha}$ into a single variable $T$).

We begin with the following lemma, which is essentially the same as 
Lemma 5 of \cite{Jantzen}.

%-------------------------------------------------------------------------------------------------------------------------------------
\begin{lem}\label{lem.move}

Let $\alpha \in \Delta_{\fn}$
 and let $\pi$ be a partition with $\pi(\alpha)>0$.
Let $\bar{\pi}$ be a partition with $\alpha$ removed from $\pi$,
 i.e.\ $\bar{\pi}(\alpha)=\pi(\alpha)-1$,
and $\bar{\pi}(\beta)=\pi(\beta)$ for all other $\beta$.
Then $x_{\alpha} x_{-\pi} \te_{\nu, i}$ is 
a linear combination of the following expressions:
\begin{align}
&(\langle \tlambda, h_{\alpha}\rangle +s) \, x_{-\bar{\pi}} \te_{\nu,i} \;,
\label{type1}\\
&(\langle \tlambda, h_{\beta}\rangle+s)\, x_{-\pi_1} \te_{\nu',i'} 
\quad \textrm{and}\quad
x_{-\pi_2} \te_{\nu'',i''} \;,
\label{type2}
\end{align}
where $\nu', \nu''\in \tlambda-Q^+(\Delta_{\fl})$, $1\le i'\le n(\nu')$, $1\le i'' \le n(\nu'')$, $\beta\in \Delta_{\fn}$,
$s\in \mathbb{C}$ and $|\pi_1|\le |\pi|-2, |\pi_2|\le |\pi|$.
\end{lem}

%-------------------------------------------------------------------------------------------------------------------------------------
\begin{proof}
We use induction on $|\pi|$.
First, note that\footnote{In this paper, the commutator $[a, b]$ is meant to be 
an anti-commutator when $a,b$ are both odd elements.
In physics literature this is sometimes denoted by $[a, b \}$.
} 
\begin{equation*}
x_{\alpha} x_{-\pi} \te_{\nu, i}=
[x_{\alpha}, x_{-\pi}] \te_{\nu, i}\pm x_{-\pi} x_{\alpha}  \te_{\nu, i} \;,
\end{equation*}
where the sign depends on the $\mathbb{Z}_2$-grading of $x_{\alpha}$ and $x_{\pi}$.
Since $x_{\alpha}\in \fn$ and $\fn$ acts trivially on $V(\tlambda)$, we have $x_{\alpha}  \te_{\nu, i}=0$
and the second term is zero.
Now the first term $[x_{\alpha}, x_{-\pi}] \te_{\nu, i}$ is a sum of the form
\beq
\pm x_{-\pi'}[x_{\alpha}, x_{-\beta} ]x_{-\pi''} \te_{\nu, i} \quad {\rm with} \quad |\pi'|+|\pi''|=|\pi|-1 \;.
\label{comp_tmp}
\eeq

If $\alpha=\beta$, then \eqref{comp_tmp} gives
$x_{-\pi'} h_{\alpha} x_{-\pi''} \te_{\nu, i}= \langle \nu-S(\pi''), h_{\alpha}\rangle  x_{-\pi'} x_{-\pi''} \te_{\nu, i}$.
Since $\nu\in \tlambda-Q^+(\Delta)$ and $S(\pi'')\in Q^+(\Delta)$,
$\langle \nu-S(\pi''),  h_{\alpha}\rangle \in \langle \tlambda-Q^+(\Delta), h_{\alpha}\rangle \subset \langle \tlambda, h_{\alpha}\rangle +\mathbb{C}$.
We therefore obtain a term of the form \eqref{type1}.

If $[x_{\alpha}, x_{-\beta} ]$ is proportional to $x_{-\beta'}$ with $\beta'\in \Delta^+$,
 then we have $x_{-\pi'}x_{-\beta'} x_{-\pi''} \te_{\nu, i}$, which is written
 as a sum of the form
 $x_{-\pi_2} \te_{\nu, i}$ with $|\pi_2| \leq |\pi|$.
 
If $[x_{\alpha}, x_{-\beta} ]$ is proportional to $x_{\beta''}$ with $\beta''\in \Delta^+$,
we can apply the assumptions of the induction to $x_{\beta''} x_{-\pi''} \te_{\nu, i}$,
to obtain expressions of either type in \eqref{type2}.
\end{proof}

%-------------------------------------------------------------------------------------------------------------------------------------

\begin{lem}
\label{lem.three} \leavevmode
(cf.\ \cite[Lemma 6]{Jantzen})

We have the following three assertions on the matrix entry $(x_{-\pi} \te_{\nu, i}, x_{-\pi'} \te_{\nu', j})$.

\begin{enumerate}

\item $(x_{-\pi} \te_{\nu, i}, x_{-\pi'} \te_{\nu', j})$
has degree equal to or smaller than ${\rm min}(|\pi|, |\pi'|)$.

\item If $|\pi|=|\pi'|$ and $(x_{-\pi} \te_{\nu, i}, x_{-\pi'} \te_{\nu', j})$
 has degree equal to $|\pi|$,
 we have $\pi=\pi', \nu=\nu'$ and $i=j$.

\item Each $(x_{-\pi} \te_{\nu, i}, x_{-\pi} \te_{\nu, i})$
 has the same leading term up to constant as the expression
\beq
\prod_{\alpha\in \Delta_{\fn}} \langle \tlambda, h_{\alpha}
 \rangle ^{\pi(\alpha)} \;.
\eeq

\end{enumerate}

\end{lem}

%-------------------------------------------------------------------------------------------------------------------------------------

\begin{proof}
We use induction on $|\pi|$. The case of $|\pi|=0$ is trivial. 
Let $\alpha\in \Delta_{\fn}$ be the first root with $\pi(\alpha)>0$.
Let $\bar{\pi}$ be a partition defined as in Lemma \ref{lem.move}.
We then have  
\beq
(x_{-\pi} \te_{\nu, i} , x_{-\pi'} \te_{\nu', j})
=
(x_{-\bar{\pi}} \te_{\nu, i} , x_{\alpha}x_{-\pi'} \te_{\nu', j})\;.
\eeq
We apply Lemma \ref{lem.move} to  $x_{\alpha}x_{-\pi'} \te_{\nu', j}$. 
Out of the resulting summands,
the types of \eqref{type2} gives, by assumption (i) of induction,
a term whose total degree is
smaller than or equal to ${\rm min}(|\pi|-1, |\pi'|)$
 or ${\rm min}(|\pi|, |\pi'|-1)$
 (recall that $\langle\tlambda,h_{\beta}\rangle$ has degree $1$).
The case of $\pi=\pi'$ is special,
 in which case the total degree is strictly smaller than $|\pi|$,
 by assumption (ii) of induction.

Let us now turn to the summand of the form \eqref{type1}.
We then need $\pi'(\alpha)>0$, and define $\bar{\pi}'$ 
from $\pi'$ as we defined $\bar{\pi}$ from $\pi$.
The $\pi'(\alpha)$ summand in total supplies
$(x_{-\bar{\pi}} e_{\nu, i} , x_{-\bar{\pi}'} e_{\nu', j})
(\langle \tlambda, h_{\alpha}\rangle +s)$
with $s\in \mathbb{C}$. The lemma now follows from the 
assumptions of the induction. 
\end{proof}

%-------------------------------------------------------------------------------------------------------------------------------------
\begin{prop}\label{prop.leading}
The leading term of the $A$-coefficient determinant $D(\lambda; \mu)_A$
is the same up to a non-zero constant as that of the following expression:
\beq
& \prod_{\nu\in \tlambda-Q^+(\Delta_{\fl})}  
\Biggl(\, \prod_{\alpha\in \overline{\Delta}_{\fn, {\sub{0}}}}
 \prod_{r=1}^{\infty}
 \langle \tlambda, h_{\alpha} \rangle^{\bfp_{\fn}(\nu-\mu - r \alpha) n(\nu)}\\
&\quad\quad \times \prod_{\alpha\in \Delta_{\sub{1}}^{+}\setminus
 \overline{\Delta}_{\sub{1}}^{+}} \prod_{r=1}^{\infty}
 \langle \tlambda, h_{\alpha} \rangle^{\bfp_{\fn}(\nu-\mu - (2r-1) \alpha) n(\nu)}  
\!
\times 
\prod_{\alpha\in \overline{\Delta}_{\sub{1}}^{+}} \langle \tlambda, h_{\alpha} \rangle^{\bfp_{\fn, \alpha}(\nu-\mu-\alpha)n(\nu)}\Biggr)\;.
\eeq
\end{prop}

\begin{proof}
%-------------------------------------------------------------------------------------------------------------------------------------
The computation below is similar to that in \cite[Lemma 10.1.3]{Musson}.

It follows from Lemma \ref{lem.three}
that the leading term of $D(\tlambda; \mu)_A$
is the same up to a constant as that of
\beqn
\prod_{\nu\in \tlambda-Q^+(\Delta_{\fl})}
\left( \prod_{\pi\in \bfP_{\fn}(\nu-\mu)} \prod_{\alpha\in \Delta_{\fn}} \langle \tlambda, h_{\alpha} \rangle^{\pi(\alpha)} 
\right)\;.
\label{lem_leading}
\eeqn

Let us simplify the expression inside the bracket in \eqref{lem_leading}.
Suppose that $\alpha \in \overline{\Delta}_{\fn, {\sub{0}}}$.
The multiplicity of $h_{\alpha}$ in  \eqref{lem_leading} is then $\sum_{\pi\in \bfP_{\fn}(\nu- \mu)} \pi(\alpha)$.
Since
\beqn
\{
\pi\in \bfP_{\fn}(\nu-\mu) | \, \pi(\alpha) =r 
\} =\bfp_{\fn}(\nu-\mu-r \alpha) - \bfp_{\fn}(\nu-\mu-(r+1) \alpha) \;,
\label{eq.comp1}
\eeqn
the multiplicity is computed to be
\beqn
\sum_{r=1}^{\infty}
r \bigl(
\bfp_{\fn}(\nu-\mu-r \alpha) - \bfp_{\fn}(\nu-\mu-(r+1) \alpha)
\bigr)
= 
\sum_{r=1}^{\infty}
\bfp_{\fn}(\nu-\mu-r \alpha) \;.
\label{eq.comp2} 
\eeqn

For $\alpha \in \Delta_{\sub{1}}^+\setminus \overline{\Delta}_{\sub{1}}^{+}$,
recall that $\beta:=2\alpha$ is an even root,
and we will take $\beta$ into account simultaneously (in the previous computation we considered $\overline{\Delta}_{\fn, {\sub{0}}}$,
 not the whole $\Delta_{\fn, {\sub{0}}}$, so there is no double counting).
 We thus have the multiplicity
\beq
\sum_{\pi\in \bfP_{\fn}(\nu-\mu)} (\pi(\alpha) + \pi(\beta)) 
=
\sum_{\pi\in \bfP_{\fn, \alpha} (\nu-\mu)}  \pi(\beta) 
+
\sum_{\pi\in \bfP_{\fn, \alpha} (\nu-\mu-\alpha)}  (1+\pi(\beta)) \;,
\eeq
where we defined
$\bfP_{\fn, \alpha}(\eta)$
as a set of partitions of $\eta$ not involving $\alpha$:
\beq
\bfP_{\fn, \alpha} (\eta):=\{
\pi \in \bfP_{\fn}(\eta) |\, \pi(\alpha)=0 
\} \;.
\eeq
Following the computations as in \eqref{eq.comp1} and \eqref{eq.comp2}, 
the first term is
\beq
\sum_{r=1}^{\infty}
r \bigl(
\bfp_{\fn,\alpha}(\nu-\mu-r \beta) - \bfp_{\fn, \alpha}(\nu-\mu-(r+1) \beta)
\bigr)
= 
\sum_{r=1}^{\infty}
\bfp_{\fn, \alpha}(\nu-\mu-r \beta) \;,
\eeq
and similarly the second term is
\beq
\bfp_{\fn, \alpha}(\nu-\mu-\alpha)
+
\sum_{r=1}^{\infty}
\bfp_{\fn, \alpha}(\nu-\mu-\alpha-r \beta) \;.
\eeq
Since 
\beq
\bfp_{\fn}(\nu-\mu+\alpha-r\beta)
 =\bfp_{\fn, \alpha}(\nu-\mu+\alpha-r\beta)
 + \bfp_{\fn, \alpha}(\nu-\mu -r\beta) \;,
\eeq
the multiplicity in the end sum up into
\beq
\sum_{r=1}^{\infty} \bfp_{\fn}(\nu-\mu-(2r-1) \alpha) \;.
\eeq

Finally, if $\alpha\in \overline{\Delta}_{\sub{1}}$ then 
\beq
\sum_{\pi\in \bfP_{\fn}(\nu-\mu)} \pi(\alpha) = \bfp_{\fn, \alpha}(\nu-\mu-\alpha) \;.
\eeq
\end{proof}

%-------------------------------------------------------------------------------------------------------------------------------------
\subsection{Position of Possible Singular Vectors}

We next turn to the positions of singular vectors.
We will see that they are highly constrained
 by the value of the Casimir operator.

%-------------------------------------------------------------------------------------------------------------------------------------
\begin{prop}\label{prop_zero}

A parabolic Verma module $M_{\mathfrak{p}}(\lambda)$
is irreducible if
  $\langle \lambda+\rho, h_{\beta}\rangle \ne \frac{1}{2}(\beta, \beta)$
for all $\beta \in Q^+(\Delta)$.

\end{prop}
%-------------------------------------------------------------------------------------------------------------------------------------

\begin{proof}
Suppose that $M_{\mathfrak{p}}(\lambda)$ is reducible.
Then there exists a highest weight vector $v\in M_{\mathfrak{p}}(\lambda)^{\lambda-\beta}$
for some $\beta\in Q^+(\Delta)$.

Let us recall that we have a Casimir element $\Omega$,
which commutes with all the generators of $\frakg$.
The Casimir element $\Omega$ on the parabolic Verma module $M_{\mathfrak{p}}(\lambda)$ takes the value \cite[Lemma 8.5.3]{Musson}
\beq
\Omega(\lambda)= (\lambda+2\rho, \lambda)  \;.
\eeq
This immediately implies $\Omega(\lambda)=\Omega(\lambda-\beta)$
and hence
\beq
(\lambda+\rho, \beta) = \frac{1}{2}(\beta, \beta) \;,
\eeq
which contradicts our initial assumption.
\end{proof}

We say $\beta \in \frakh^*$ is a {\it quasi-root} 
if $\beta = r \alpha$ for some $r\in \mathbb{Z}$ 
and $\alpha \in \Delta^+$.
By combining Propositions \ref{prop.leading} and \ref{prop_zero},
we obtain the following:

%-------------------------------------------------------------------------------------------------------------------------------------
\begin{prop}\label{prop_plane}

Up to a non-zero constant, the determinant $D(\tlambda; \mu)_A$
is a product of degree-one polynomials of the form
\beqn
F_{\beta}:= \langle \tlambda+\rho, h_{\beta} \rangle- \frac{1}{2}(\beta, \beta) \;,
\label{eq.prop_zero}
 \eeqn
where $\beta\in \frakh^*$ is a quasi-root.
\end{prop}

\begin{proof}
Proposition~\ref{prop_zero} implies $D(\lambda;\mu)\neq 0$ if 
$\langle \lambda+\rho, h_{\beta} \rangle- \frac{1}{2}(\beta, \beta)\neq 0$
 for all $\beta\in Q^+(\Delta)$.
Therefore, $D(\tlambda;\mu)_A$ divides a product of $F_\beta$
 $(\beta\in Q^+(\Delta))$
 and hence we can write
 $D(\tlambda;\mu)_A=c\prod_{i=1}^n F_{\beta_i}$,
 where $c\in \bC\setminus\{0\}$ and $\beta_i\in Q^+(\Delta)$.
Then by Proposition~\ref{prop.leading}, $\beta_i$ have to be quasi-roots.
\end{proof}

%-------------------------------------------------------------------------------------------------------------------------------------

\subsection{Jantzen Filtration}

The remaining task is to compute the multiplicities of the factor 
\eqref{eq.prop_zero}. Since polynomials
$F_{\beta}$ for different $\beta$'s may have the same leading term (up to a constant multiplication), Proposition \ref{prop.leading} is not enough to determine the multiplicities.

%-------------------------------------------------------------------------------------------------------------------------------------

As explained before, we are interested in the order of zeros.
Suppose that we want to calculate the order of the zero for the factor $p$,
where $p$ is some degree-one polynomial of the variables $\{ T_{\alpha} \}_{\alpha\in \Pi\setminus \Pi_{\fl}}$.
The problem is then to compute the value of the $p$-adic valuation of the determinant,
where the valuation is defined by  $v_p(p^n a)=n$ for $a\in A$ and $a$ is not divisible by $p$.
Such a valuation can be evaluated with the help of the following theorem,
 which introduces the so-called {\it Jantzen filtration}:
\beq
M=M(0)  \supset M(1) \supset M(2) \supset \cdots  \:.
\eeq

%-------------------------------------------------------------------------------------------------------------------------------------
\begin{lem} \cite[Lemma 3]{Jantzen} \label{lem.J}

Let $A'$ be a principal ideal domain, and $p\in A'$ a prime element.
Let $K=A'/pA'$ be the quotient field and $\varphi: A'\to K$ the canonical map. 
We write $v_p$ for the $p$-adic valuation of $A'$.
Suppose that $M$ is a free $A'$-module of finite rank
 with a symmetric bilinear form $(-, -)$ with values in $A'$.
We also write $\varphi$ for the canonical map $M\to M/pM$.
Let $D$ be the determinant of $(-, -)$ with respect to a basis of $M$.
For all $n\in \mathbb{N}$ set 
\beq
M(n):=\{ x\in M |\, (x,m) \subset A'p^n \} \;.
\eeq
Then $M(n)$ is a lattice inside $M$ and for $n\ge 1$ $M(n)/pM(n-1)$ is a $K$-vector space.
If $D\ne 0$, then we have
\beq
v_p(D)=\sum_{n>0} \dim_K\left(M(n)/pM(n-1)\right)=\sum_{n>0} \dim_K\left(\varphi(M(n))\right)  \;.
\eeq
\end{lem}

\begin{proof}
See \cite{Jantzen}.
\end{proof}

%%-------------------------------------------------------------------------------------------------------------------------------------

In order to apply Lemma \ref{lem.J} to our problem, 
 we localize $A$ at a prime ideal $(p)$ and set $A':= A_{(p)}$.
Since there is a canonical injective homomorphism from $A$ to $A'$, we have the equation
 $v_p(D(\lambda;\mu)_A) = v_p(D(\lambda;\mu)_{A'})$.
In order to avoid clutter in the notation, we often use the same symbol (e.g.\ $\tlambda$) for different coefficient rings $A$ and $A'$.
We can then apply Lemma \ref{lem.J} to the ring $A'$, with the quotient field
$K=A'/p A'$ and a canonical homomorphism $\varphi: A'\to K$.

%-------------------------------------------------------------------------------------------------------------------------------------

We are now ready to state the consequences of Lemma \ref{lem.J}.
In the following lemma $L(\tlambda)_K$ denotes the unique
 irreducible quotient module of $M_{\mathfrak{p}}(\tlambda)_K$.
We define
\beq
v_p(D(\tlambda)_{A}):=\sum_{\mu\leq \tlambda} v_p(D(\tlambda;\mu)_{A}) e^{\mu} \;,\quad
v_p(D(\tlambda)_{A'}):=\sum_{\mu\leq \tlambda} v_p(D(\tlambda;\mu)_{A'}) e^{\mu} \;.
\eeq
They are identified with each other by a canonical homomorphism $A\to A'$.

%-------------------------------------------------------------------------------------------------------------------------------------

\begin{lem} \leavevmode
(cf.\ \cite[Satz 1]{Jantzen})
\label{lem.ab}

Suppose $\tlambda\in P^+(\Delta_{\fl})_A$. 
Then 
there exist $a(\tlambda, \delta), b(\tlambda, \delta) \in \mathbb{N}$
 for $\delta>0$ such that
\begin{align}
{\rm ch} \, M_{\frakp}(\varphi(\tlambda))_K
=
{\rm ch} \, L(\varphi(\tlambda))_K
+
\sum_{\delta>0}
 a(\tlambda, \delta) \, {\rm ch} \, L(\varphi(\tlambda)-\delta)_K \;,
\label{sum1}
\end{align}
and
\begin{align}
\label{sum2}
\varphi\bigl(v_p(D(\tlambda)_{A'})\bigr)
= \sum_{\delta>0} b(\tlambda, \delta)
 \, {\rm ch}\, L(\varphi(\tlambda)-\delta)_K \;.
\end{align}
Moreover, $a(\tlambda, \delta)> 0$ implies 
 $\Omega(\varphi(\tlambda))=\Omega(\varphi(\tlambda)-\delta)$.
Further, $a(\tlambda, \delta)> 0$ exactly when $b(\tlambda, \delta)> 0$,
and $a(\tlambda, \delta)\le b(\tlambda, \delta)$ for all $\delta>0$.
\end{lem}

%-------------------------------------------------------------------------------------------------------------------------------------

\begin{proof}
By considering irreducible decomposition of the module $M_{\frakp}(\varphi(\tlambda))_K$ we get the expression \eqref{sum1}.
Note that 
\beq
\dim M_{\frakp}(\varphi(\tlambda))_K^{\varphi(\tlambda)}=1
\eeq 
implies $M_{\frakp}(\varphi(\tlambda))_K$ contains exactly one copy of 
$L(\varphi(\tlambda))_K$.
We can also see that $a(\tlambda, \delta)> 0$ implies 
 $\Omega(\varphi(\tlambda))=\Omega(\varphi(\tlambda)-\delta)$
 as in the proof of Proposition~\ref{prop_zero}.

Let us define $M_{\frakp}(\tlambda)_{A'}^{\mu}(n)$ as in  Lemma \ref{lem.J}
 and define $M_{\frakp}(\tlambda)_{A'}(n)$ by 
\beq
M_{\mathfrak{p}}(\tlambda)_{A'}(n) = \bigoplus_{\mu} M_{\mathfrak{p}}(\tlambda)_{A'}^{\mu}(n) \;.
\eeq
Then $M_{\frakp}(\tlambda)_{A'}(n)$ is $\mathfrak{g}_{A'}$-stable.
It follows from Lemma \ref{lem.J}  that 
\beq
v_p(D(\tlambda;\mu)_{A'})= 
\sum_{n>0} \dim_{K} \bigl(M_{\mathfrak{p}}(\tlambda)_{A'}^{\mu}(n)/pM_{\mathfrak{p}}(\tlambda)_{A'}^{\mu}(n-1)\bigr) \;.
\eeq
This can be rewritten as 
\beqn\label{eq.sum_n}
\varphi\bigl(v_p(D(\tlambda)_{A'})\bigr)= 
\sum_{n>0} \textrm{ch}\bigl(M_{\mathfrak{p}}(\tlambda)_{A'}(n)/pM_{\mathfrak{p}}(\tlambda)_{A'}(n-1)\bigr) \;.
\eeqn
By decomposing the $\mathfrak{g}_K$-module 
 $M_{\mathfrak{p}}(\tlambda)_{A'}(n)/pM_{\mathfrak{p}}(\tlambda)_{A'}(n-1)$ into irreducible modules we get the expression \eqref{sum2}.

Now a crucial property of the Jantzen filtration is that we have
\beqn
\textrm{ch}\bigl(M_{\mathfrak{p}}(\tlambda)_{A'}(1)/pM_{\mathfrak{p}}(\tlambda)_{A'}\bigr)
=
\textrm{ch}\bigl(M_{\frakp}(\varphi(\tlambda))_K\bigr)- \textrm{ch}\bigl(L(\varphi(\tlambda))_K\bigr) \;.
\label{eq.radical}
\eeqn
To show this, note that $M_{\mathfrak{p}}(\tlambda)_{A'}(1)/pM_{\mathfrak{p}}(\tlambda)_{A'}$ is by definition 
the radical of the contravariant form of $M^{\frakp}(\varphi(\tlambda))_K$. Since the radical is the largest submodule, and since $L(\varphi(\tlambda))_K$ is its quotient module,
 \eqref{eq.radical} follows.
Since \eqref{eq.radical} is the $n=1$ term of the sum \eqref{eq.sum_n}, we obtain $a(\tlambda, \delta)\le b(\tlambda, \delta)$.
Moreover, we have for all $n>0$ a natural surjective morphism
\begin{align*}
M_{\mathfrak{p}}(\tlambda)_{A'}(n)/pM_{\mathfrak{p}}(\tlambda)_{A'}(n)
 \to 
M_{\mathfrak{p}}(\tlambda)_{A'}(n)/pM_{\mathfrak{p}}(\tlambda)_{A'}(n-1)
\end{align*}
and the equality
\begin{align*}
\textrm{ch}\bigl(
M_{\mathfrak{p}}(\tlambda)_{A'}(n)/pM_{\mathfrak{p}}(\tlambda)_{A'}(n)
\bigr)
= \textrm{ch}\bigl( M_{\mathfrak{p}}(\varphi(\tlambda))_{K} \bigr)\;.
\end{align*}
This equality implies that
 $M_{\mathfrak{p}}(\tlambda)_{A'}(n)/pM_{\mathfrak{p}}(\tlambda)_{A'}(n)$
 have the same composition factors
 as $M_{\mathfrak{p}}(\varphi(\tlambda))_{K}$.
Then by the above surjective morphism 
 any composition factor of 
 $M_{\mathfrak{p}}(\tlambda)_{A'}(n)/pM_{\mathfrak{p}}(\tlambda)_{A'}(n-1)$
 is one of $M_{\mathfrak{p}}(\varphi(\tlambda))_{K}$,
 proving that 
 $b(\tlambda, \delta)> 0$ only when $a(\tlambda, \delta)> 0$.
\end{proof}

\begin{lem}\label{lem.invert}
Let $\mu\in \frakh^*_K$.
Then we have 
\begin{align*}
{\rm ch} \, L(\mu)_K
={\rm ch} \, M(\mu)_K
+\sum_{\substack{\delta>0 \\ \Omega(\mu)=\Omega(\mu-\delta)}}
 c(\mu, \delta) \, {\rm ch} \, M(\mu-\delta)_K
\end{align*}
for some $c(\mu,\delta)\in \bZ$.
Here, the right hand side is possibly an infinite sum.
\end{lem}

\begin{proof}
As in \eqref{sum1} we have
\begin{align}\label{eq.chML}
{\rm ch} \, M(\nu)_K ={\rm ch} \, L(\nu)_K
+\sum_{\substack{\delta>0 \\ \Omega(\nu)=\Omega(\nu-\delta)}}
 a'(\nu, \delta) \, {\rm ch} \, L(\nu-\delta)_K \;.
\end{align}
The first term on the right hand side
 is the character of the irreducible module $L(\nu)_K$
with the same weight $\nu$ as that on the left hand side,
 with coefficient one.
The second sum contains the characters of irreducible modules,
 with smaller highest weights $\nu-\delta$ with $\delta>0$.
We can therefore invert \eqref{eq.chML}, to obtain the lemma.
\end{proof}

%-------------------------------------------------------------------------------------------------------------------------------------

\subsection{Determination of Order of Zeros}

We shall now determine the multiplicity of each factor $F_{\beta}$,
with the help of the Jantzen filtration.

We take non-zero degree-one homogeneous polynomials $p_1, \ldots, p_n$ such that for each $\alpha \in \Delta_{\fn}$ there exists exactly one $i=i(\alpha)$
 with $\langle \tlambda+\rho, h_{\alpha} \rangle = s(p_i+t)$ and 
 $s, t \in \mathbb{C}, s\ne 0$.
Let 
\beq
\Delta^i_{\fn}:=\{\alpha\in \Delta_{\fn}\; |\; i=i(\alpha) \}
\eeq
and let 
\beq
\bDelta_{\fn,{\sub{0}}}^{i}
 :=\bDelta_{\fn,{\sub{0}}} \cap \Delta^i_{\fn} \;, \quad
\Delta_{\sub{1}}^{+,i}:=\Delta_{\sub{1}}^+ \cap \Delta^i_{\fn}\;, \quad
\bDelta_{\sub{1}}^{+,i}:=\bDelta_{\sub{1}}^+ \cap \Delta^i_{\fn}\;.
\eeq

%-------------------------------------------------------------------------------------------------------------------------------------
\begin{lem}\label{lem.Weyl}
We have
\beq
w \bigl( \bDelta_{\fn,{\sub{0}}}^{i}\bigr)= \bDelta_{\fn,{\sub{0}}}^{i} \;,
\quad
w\bigl( \Delta_{\sub{1}}^{+,i} \setminus \bDelta_{\sub{1}}^{+,i} \bigr)
 =  \Delta_{\sub{1}}^{+,i} \setminus \bDelta_{\sub{1}}^{+,i} \;,
\quad
w\bigl(\bDelta_{\sub{1}}^{+,i}\bigr)=\bDelta_{\sub{1}}^{+,i} \;,
\label{weyl_inv}
\eeq
for $w\in W_{\fl}$.
\end{lem}
%-------------------------------------------------------------------------------------------------------------------------------------

\begin{proof}
For $\alpha\in \Delta_{\fn}^i$ and $w\in W_{\fl}$ we have
\beq
\langle \tlambda+ \rho, h_{w(\alpha)} \rangle=
\langle w^{-1}(\tlambda+ \rho), h_{\alpha}\rangle
\in
\langle  \tlambda+ \rho+Q^+(\Delta_{\fl}), h_{\alpha} \rangle
\subset
\langle \tlambda+ \rho, h_{\alpha} \rangle +\mathbb{C} \;.
\eeq
\end{proof}

%-------------------------------------------------------------------------------------------------------------------------------------

From Proposition \ref{prop_plane} and the definition above
we learn that the only irreducible polynomials which could divide 
the determinant $D(\tlambda;\mu)_A$
are $p_i+t$ for $t\in \mathbb{C}$.
This means that we have
\beq
 D(\tlambda; \mu)_A = c \prod_{i=1}^n
 \prod_{t\in \mathbb{C}} (p_i+t)^{v(i, t; \mu)}\;,
\eeq
where $c\in \mathbb{C}\setminus \{0\}$, and the power
\beq
v(i, t; \mu):=v_{p_i+t}( D(\tlambda; \mu)_A) \in \mathbb{N}
\eeq
is the multiplicity we wish to determine.
We note that when $\mu$ is fixed, 
 $v(i, t; \mu)=0$ for all but finitely many pairs $(i,t)$.

We have not yet determined $v(i, t; \mu)$,
 but we already know their sum for $t\in \bC$
 from the leading term of the determinant.
For our later purposes it is useful to define the combination
\beq
v(i, t):=\sum_{\mu} v(i, t; \mu)\, e^{\mu}\;.
\eeq

%-------------------------------------------------------------------------------------------------------------------------------------
\begin{lem}\label{lem.sum}
We have 
\begin{align}
\begin{split}\label{lem.vi}
\sum_{t\in \mathbb{C}} v(i,t)
=  \sum_{\alpha \in \bDelta_{\fn,{\sub{0}}}^{i}} \sum_{r=1}^{\infty}
  \chi^{\mathfrak{p}}(\tlambda-r\alpha)
&+ \sum_{\alpha \in  \Delta_{\sub{1}}^{+,i} \setminus \bDelta_{\sub{1}}^{+,i}}
 \sum_{r=1}^{\infty}  \chi^{\mathfrak{p}}(\tlambda-(2r-1)\alpha)
 \\
&+ \sum_{\alpha \in \bDelta_{\sub{1}}^{+,i}} \  
 \chi^{\mathfrak{p}}_{\alpha}(\tlambda-\alpha)
\;.
\end{split}
\end{align}
\end{lem}
%-------------------------------------------------------------------------------------------------------------------------------------

\begin{proof}
From Proposition \ref{prop.leading} we already know that
\begin{align}\label{eq.vi}
\begin{split}
\sum_{t\in \mathbb{C}} v(i,t)
&=\sum_{\mu} \sum_{\nu\in \tlambda-Q^+(\Delta_{\fl})}
 \sum_{\alpha \in \bDelta_{\fn,{\sub{0}}}^{i}}
 \sum_{r=1}^{\infty} \bfp_{\fn}(\nu-\mu-r \alpha) \,n(\nu)\, e^{\mu} \\ 
&\qquad +\sum_{\mu} \sum_{\nu\in \tlambda-Q^+(\Delta_{\fl})}
 \sum_{\alpha \in \Delta_{\sub{1}}^{+,i} \setminus \bDelta_{\sub{1}}^{+,i}} 
 \sum_{r=1}^{\infty} \bfp_{\fn}(\nu-\mu-(2r-1) \alpha)\, n(\nu) \,e^{\mu} \\ 
&\qquad +\sum_{\mu} \sum_{\nu\in \tlambda-Q^+(\Delta_{\fl})}
 \sum_{\alpha \in \bDelta_{\sub{1}}^{+,i}}
 \bfp_{\fn, \alpha}(\nu-\mu-\alpha) \,n(\nu) \,e^{\mu}
\;.
\end{split}
\end{align}
For $\alpha \in \bDelta_{\sub{0}}^+$ we have 
\begin{align*}
&\sum_{\mu}
\sum_{\nu\in \tlambda-Q^+(\Delta_{\fl})}
 \sum_{r=1}^{\infty}
 \bfp_{\fn}(\nu-\mu-r \alpha) \,n(\nu)\, e^{\mu}
\\
&
\quad
=\Biggl(\sum_{\nu\in \tlambda-Q^+(\Delta_{\fl})}  \sum_{\mu}
 \bfp_{\fn}(\nu-\mu) \,n(\nu)\, e^{\mu} \Biggr)
\times \sum_{r=1}^{\infty} e^{-r\alpha}\\
&\quad=\mathfrak{P}_{\fn} \, \textrm{ch}\, V(\tlambda)
 \times \sum_{r=1}^{\infty} e^{-r\alpha}\\
&\quad= \mathfrak{P} 
 \Biggl( \sum_{w\in W_{\fl}} 
 \det(w) e^{w.\tlambda}  \Biggr)
 \times \sum_{r=1}^{\infty} e^{-r\alpha} \\
&\quad=  \sum_{r=1}^{\infty} \sum_{w\in W_{\fl}}
 \det(w) \, {\rm ch}\, M(w.\tlambda-r\alpha)\;,
\end{align*}
where in the fourth line we used
 the Weyl character formula \eqref{eq.Weyl_char} and the factorization \eqref{tripleP}.
 Now with the help of Lemma \ref{lem.Weyl} we can rewrite the first term of the right hand side of \eqref{eq.vi} as 
\begin{align*}
\sum_{\alpha \in  \bDelta_{\fn,{\sub{0}}}^{i}}
\sum_{r=1}^{\infty} \sum_{w\in W_{\fl}}
 \det(w) \, {\rm ch}\, M(w.\tlambda-r\alpha)
&=
\sum_{r=1}^{\infty} \sum_{w\in W_{\fl}}
\sum_{\alpha \in  \bDelta_{\fn,{\sub{0}}}^{i}}
 \det(w) \, {\rm ch}\, M(w.\tlambda-rw\alpha)\\
&=\sum_{\alpha \in  \bDelta_{\fn,{\sub{0}}}^{i}} \sum_{r=1}^{\infty}
  \chi^{\mathfrak{p}}(\tlambda-r\alpha)\; ,
\end{align*}
where in the last step we used the definition \eqref{chi_def}
 and $w.(\tlambda-r\alpha)=w.\tlambda-rw\alpha$.

We can rewrite the sums for
 $\alpha \in \Delta_{\sub{1}}^{+,i} \setminus \bDelta_{\sub{1}}^{+,i}$
 and $\alpha \in  \bDelta_{\sub{1}}^{+,i}$ in a similar way
 and obtain \eqref{lem.vi}.
\end{proof}

Lemma \ref{lem.sum} expresses the sum $\sum_{t\in \mathbb{C}} v(i, t)$
 as a linear combination of $\chi^{\frakp}(\tlambda-s\alpha)$
 and $\chi^{\frakp}_{\alpha}(\tlambda-\alpha)$.
Then it is written also as a linear combination of
 ${\rm ch}\, M(\tlambda-\delta)$ for $\delta>0$.
We now claim that not only their sum $\sum_{t\in \mathbb{C}} v(i, t)$
 but also each $v(i, t)$ can be written as a linear combination
 of characters ${\rm ch}\, M(\tlambda-\delta)$.  
Moreover, we get an additional condition on $\delta$
 for ${\rm ch}\, M(\tlambda-\delta)$ to appear in the expression of $v(i, t)$,
 by working out the value of Casimir element. 
In the following proposition we put $p=p_i+t$
 and use notation (e.g.\ $\varphi$) in Lemma~\ref{lem.ab}.

%-------------------------------------------------------------------------------------------------------------------------------------

\begin{lem}\label{lem.new}
$v(i, t) \in \mathbb{Z}\langle \frakh^{*}_A \rangle$ is a $\bZ$-linear combination of ${\rm ch}\, M(\tlambda-\delta)_A$, where $\delta>0$ and $\Omega(\varphi(\tlambda))=\Omega(\varphi(\tlambda)-\delta)$.

\end{lem}

%-------------------------------------------------------------------------------------------------------------------------------------

\begin{proof}
Since $\varphi$ sends $\tlambda-Q^+(\Delta)$ bijectively onto $\varphi(\tlambda)-Q^+(\Delta)$, it is enough to show that $\varphi(v(i,t))$ is a $\bZ$-linear sum of ${\rm ch}\, M(\varphi(\tlambda)-\delta)_K$ for $\delta>0$ and $\Omega(\varphi(\tlambda))=\Omega(\varphi(\tlambda)-\delta)$.
We have $\varphi(v(i,t))=\varphi(v_{p}(D(\tlambda)_{A}))$ and the decomposition of $\varphi(v_{p}(D(\tlambda)_{A}))$
as in \eqref{sum2}. Suppose that $b(\tlambda, \delta)> 0$.
Then $a(\tlambda, \delta)> 0$ and hence $\Omega(\varphi(\tlambda))= \Omega(\varphi(\tlambda)-\delta)$ by Lemma~\ref{lem.ab}.
We therefore have 
\beq
\varphi(v_{p}(D(\tlambda)_{A})) 
=
\sum_{\substack{\delta>0 \\ \Omega(\varphi(\tlambda))= \Omega(\varphi(\tlambda)-\delta)}} b(\tlambda, \delta)\, \textrm{ch}\, L(\varphi(\tlambda)-\delta)_K \;.
\eeq
By Lemma~\ref{lem.invert},
 we can express $\textrm{ch}\, L(\varphi(\tlambda)-\delta)_K$ as
 a $\bZ$-linear sum of 
 $\textrm{ch}\, M(\varphi(\tlambda)-\delta')_K$ for $\delta' \geq \delta$
 such that
 $\Omega(\varphi(\tlambda)-\delta)=\Omega(\varphi(\tlambda)-\delta')$.
\end{proof}

%-------------------------------------------------------------------------------------------------------------------------------------

From Lemma \ref{lem.new} we find
\beq
v(i,t) =
 \sum_{\substack{\delta>0\\ \Omega(\varphi(\tlambda)) =\Omega(\varphi(\tlambda)-\delta)}} \,
c_t(\tlambda, \delta) \, {\rm ch}\, M (\tlambda-\delta)_A
\eeq
with $c_t (\tlambda, \delta)\in \mathbb{Z}$.

%-------------------------------------------------------------------------------------------------------------------------------------

\begin{lem}\label{lem.divide}
Suppose that we have $c_t(\tlambda, \delta) \ne 0$.
Then $p_i+t \sim  \Omega(\tlambda)-\Omega(\tlambda-\delta)$,
 where $\sim$ means that two polynomials are equal up
 to a constant multiplication.
\end{lem}

%-------------------------------------------------------------------------------------------------------------------------------------

\begin{proof}
Since $c_t(\tlambda, \delta) \ne 0$, we have 
\beq
\varphi( \Omega(\tlambda) - \Omega(\tlambda - \delta) )
=\Omega(\varphi(\tlambda)) -\Omega(\varphi(\tlambda) - \delta)  =0 \;.
\eeq
This means $p_i+t$
 divides $\Omega(\varphi(\tlambda)) -\Omega(\varphi(\tlambda) - \delta)$.
Since both $p_i+t$
 and $\Omega(\varphi(\tlambda)) -\Omega(\varphi(\tlambda) - \delta)$ 
 have degree-one in $\{ T_{\alpha} \}_{\alpha\in \Pi\setminus \Pi_{\fl}}$, 
 the two polynomials have to be proportional.
\end{proof}

%-------------------------------------------------------------------------------------------------------------------------------------

Lemma~\ref{lem.divide} implies that 
 for a fixed $\delta$
 there exists at most only one $t\in \mathbb{C}$ such that 
 $c_t(\tlambda, \delta)\ne 0$. 
In other words, if $c_t(\tlambda, \delta), c_{t'}(\tlambda, \delta)\ne 0$
 then $t=t'$.
Indeed, 
 since $c_t(\tlambda, \delta), c_{t'}(\tlambda, \delta)\ne 0$,
 we have from Lemma \ref{lem.divide}
that 
\beq
p_i+t \sim \Omega(\tlambda)- \Omega(\tlambda-\delta) \sim p_i+t' \;,
\eeq
 which immediately gives $t=t'$.

%-------------------------------------------------------------------------------------------------------------------------------------

This observation and Lemma~\ref{lem.cancel}
 show that the sum $\sum_t v(i,t)$ 
 obtained previously is enough to recover $v(i,t)$:

\begin{lem}\label{lem.unique}
Fix $1\leq i\leq n$.
For $t\in \bC$ suppose that $u(i,t), u'(i,t)\in \bZ\langle \frakh^*_A\rangle$
 are $\bZ$-linear sums of ${\rm ch}\, M(\tlambda-\delta)_A$
 such that $\delta>0$ and $p_i+r\sim \Omega(\tlambda)-\Omega(\tlambda-\delta)$.
Suppose that the sums $\sum_{t\in \bC} u(i,t)$ and $\sum_{t\in \bC} u'(i,t)$
 are well-defined and $\sum_{t\in \bC} u(i,t) = \sum_{t\in \bC}u'(i,t) $.
Then $u(i,t)=u'(i,t)$ for all $t\in \bC$.
\end{lem}

\begin{proof}[Proof of Theorem \ref{thm.main}]
Let us define
\begin{align*}
u(i,t):= 
 \sum_{(\alpha,r)} \chi^{\mathfrak{p}}(\tlambda-r\alpha)
+ \sum_{(\alpha,r)}  \chi^{\mathfrak{p}}(\tlambda-(2r-1)\alpha) 
+ \sum_{\alpha} \  \chi^{\mathfrak{p}}_{\alpha}(\tlambda-\alpha) \;.
\end{align*}
Here, $(\alpha,r)$ runs over
 $\alpha\in \bDelta_{\fn,{\sub{0}}}^i$ and $r\in \bZ_{> 0}$
 such that $p_i+t\sim \Omega(\tlambda)-\Omega(\tlambda-r\alpha)$
 for the first term;
 $(\alpha,r)$ runs over
 $\alpha\in {\Delta}_{\sub{1}}^{+,i}\setminus\overline{\Delta}_{\sub{1}}^{+,i}$
 and $r\in \bZ_{> 0}$
 such that $p_i+t\sim \Omega(\tlambda)-\Omega(\tlambda-(2r-1)\alpha)$
 for the second term;
 and $\alpha$ runs over $\alpha\in \overline{\Delta}_{\sub{1}}^{+,i}$
 such that $p_i+t\sim \Omega(\tlambda)-\Omega(\tlambda-\alpha)$.
Then $u(i,t)$ is written as a $\bZ$-linear sum of
 ${\rm ch}\, M(\tlambda-\delta)_A$ such that $\delta>0$
 and $p_i+t\sim \Omega(\tlambda)-\Omega(\tlambda-\delta)$.
 Moreover, we have $\sum_{t\in \bC} u(i,t)= \sum_{t\in \bC} v(i,t)$ 
 thanks to Lemma \ref{lem.sum}.
Hence Lemma~\ref{lem.unique} implies $v(i,t)=u(i,t)$.

Since 
\beq
\Omega(\tlambda)-\Omega(\tlambda-r\alpha)
=2r\left( (\tlambda+\rho, \alpha) - \frac{r}{2}(\alpha, \alpha)\right) \;, 
\eeq
we conclude that the determinant formula~\eqref{fermionic_determinant} holds.
\end{proof}

\begin{rem}
We did not include the case $\mathfrak{psl}(n|n), \mathfrak{sl}(n|n)$
 in our argument.
However, determinant formulas for these cases can be easily deduced from
 that for $\mathfrak{gl}(n|n)$. 

To see this, note first that the only differences between
 these Lie superalgebras are the Cartan subalgebras.

Suppose that $\frakp$ is a parabolic subalgebra of $\mathfrak{sl}(n|n)$ and
 let $\frakp'$ be a parabolic subalgebra of $\mathfrak{gl}(n|n)$
 such that $\frakp'\cap\mathfrak{sl}(n|n)=\frakp$.  
Then a parabolic Verma module
 $M_{\frakp}(\lambda)$ of $\mathfrak{sl}(n|n)$ is isomorphic to
 a parabolic Verma module $M_{\frakp'}(\lambda')$
 of $\mathfrak{gl}(n|n)$, where $\lambda'$ is an extension of $\lambda$
 to a Cartan subalgebra of $\mathfrak{gl}(n|n)$.
They have the same contravariant form $(-,-)$ and 
 hence we obtain the same determinant formula for $\mathfrak{sl}(n|n)$.

Similarly, a parabolic Verma module $M_{\Bar{\frakp}}(\Bar{\lambda})$ of
 $\mathfrak{psl}(n|n)$ is isomorphic
 to a parabolic Verma module $M_{\frakp}(\lambda)$ of
 $\mathfrak{sl}(n|n)$,
 where $\frakp$ is the inverse image of $\Bar{\frakp}$ by the natural map
 $\mathfrak{sl}(n|n)\to \mathfrak{psl}(n|n)$ and $\lambda$ is the pull-back
 of $\Bar{\lambda}$.
Hence the same determinant formula also holds for $\mathfrak{psl}(n|n)$.
\end{rem}

%%%%%%%%%%%%%%%%%%%%%%%%%%%%%%%%%%%%%%%%%%%%%%%%
\section{Irreducibility Criteria}\label{sec.simplicity}

Let us next study irreducibility criteria,
i.e.\ conditions for
 the non-existence of singular vectors (a.k.a.\ null states).
Most of the results and arguments in this section
 are again parallel to Jantzen~\cite{Jantzen}.

Since a parabolic Verma module $M_{\frakp}(\lambda)$ is irreducible
 if and only if the contravariant form $(-,-)$ is non-degenerate,
 we can get irreducibility criteria from the determinant formula.
A direct consequence of the determinant formula is that
 a parabolic Verma module $M_{\frakp}(\lambda)$ is irreducible  
 if the following set $\Psi_{\lambda}$ is empty: 
\begin{align*}
\begin{split}
\Psi_{\lambda}:=
\Psi_{\lambda, {\rm non\mathchar`-iso}} \cup
\Psi_{\lambda, {\rm iso}} \;,
\end{split}
\end{align*}
where
\begin{align*}
\begin{split}
\Psi_{\lambda, {\rm non\mathchar`-iso}}:=
&\left\{
\alpha \in \bDelta_{\fn, {\sub{0}}} \Big|
\, n_{\alpha}:=
 \frac{2( \lambda+\rho, \alpha )}{(\alpha, \alpha)} \in \mathbb{Z}_{>0}
\right\} \\
&\cup
\left\{
\alpha \in \Delta_{\sub{1}}^+ \setminus  \overline{\Delta}_{\sub{1}}^+ \Big|
\, n_{\alpha}:=\frac{2( \lambda+\rho, \alpha )}{(\alpha, \alpha)}
 \in 2\mathbb{N}+1 \right\}  \;,
\end{split}
\end{align*}
and 
\begin{align*}
\begin{split}
\Psi_{\lambda, {\rm iso}}:=
\left\{
\alpha \in \overline{\Delta}_{\sub{1}}^+ \Big|
\, ( \lambda+\rho, \alpha )=0 
\right\}  \;.
\end{split}
\end{align*}

%----------------------------------------------------------------------------------------------------

Even if this set is not empty, however, we cannot in general conclude that the 
parabolic Verma module is reducible---there could be 
cancellations in the exponents in \eqref{fermionic_determinant}.

As we have seen before, 
the complication here is
 that several polynomials of the form
 $(\tlambda + \rho, \alpha)-r(\alpha,\alpha)$
 may coincide up to constant multiplication.
Recall that in the previous section we defined polynomials
 $p_i$ for $1\leq i\leq n$ and 
 $i(\alpha)$ for $\alpha\in \Delta_{\fn}$ such that 
 $\langle \tlambda + \rho, h_\alpha\rangle=s(p_{i(\alpha)}(\tlambda)+t)$
 for some $s,t\in\bC$.
Let us consider
 the valuation with respect to the prime element $p_i(\tlambda)+t$.
It follows that
 $D(\tlambda;\mu)_A\neq 0$ for all $\mu\leq \tlambda$ if and only if 
 $v_{p_i(\tlambda)+ t}(D(\tlambda)_A)=0$
 for all $i$ and $t$.
For $1\leq i\leq n$ take $\alpha\in\Delta_{\fn}$ such that $i(\alpha)=i$
 and define 
\begin{align*}
\Delta^i:= (\mathbb{Q} \Delta_{\fl} + \mathbb{Q} \alpha)\cap \Delta\; .
\end{align*}
The set $\Delta^i$ does not depend on the choice of $\alpha$.
We define moreover
\begin{align*}
\Psi_{\lambda}^{i}:=\Psi_{\lambda}\cap\Delta^i ,\quad  
\Psi_{\lambda, {\rm non\mathchar`-iso}}^{i}
 :=\Psi_{\lambda, {\rm non\mathchar`-iso}}\cap\Delta^i, \quad 
\Psi_{\lambda, {\rm iso}}^{i}:=\Psi_{\lambda, {\rm iso}}\cap\Delta^i \;.
\end{align*}

Consider a factor $(\tlambda +\rho, \alpha)-\frac{r}{2}(\alpha,\alpha)$ in the
determinant in the coefficient ring $A$.
This factor is equal to 
 $p_i(\tlambda)-p_i(\lambda)$ up to constant multiplication
 if and only if $\alpha\in \Psi_{\lambda}^i$ and
 $(\lambda +\rho, \alpha)-\frac{r}{2}(\alpha,\alpha)=0$.
The latter gives 
 $r=-n_{\alpha}$ when $\alpha$ is non-isotropic.
For isotropic $\alpha\in \Psi_{\lambda, {\rm iso}}^{i}$
 this is automatic.
Hence 
\begin{align*}
v_{p_i(\tlambda)-p_i(\lambda)}(D(\tlambda)_A)
=\sum_{\alpha \in \Psi_{\lambda, {\rm non\mathchar`-iso}}^{i}}
 \chi^{\mathfrak{p}}(\tlambda-n_{\alpha}\alpha)
+\sum_{\alpha \in \Psi_{\lambda, {\rm iso}}^{i}}
 \chi^{\mathfrak{p}}_{\alpha}(\tlambda-\alpha) \; .
\end{align*}
Since 
\beq
s_{\alpha}.\lambda
 := \lambda - \langle \lambda+\rho, \alpha^{\vee} \rangle \alpha
 = \lambda-n_{\alpha}\alpha 
\eeq
for a non-isotropic root $\alpha$, 
we obtain the following irreducibility criterion:
%----------------------------------------------------------------------------------------------------
\begin{thm}[Irreducibility Criterion (1)]\label{prop_1}
Let $\lambda\in P^+(\Delta_{\fl})$.
The parabolic Verma module $M_{\mathfrak{p}}(\lambda)$ is irreducible 
 if and only if the following condition holds:
for all $1\leq i\leq n$
\begin{empheq}[box={\mybluebox[7pt]}]{equation}
\label{eq.prop_1}
\quad 
\sum_{\alpha \in \Psi_{\lambda, {\rm non\mathchar`-iso}}^{i}}
 \chi^{\mathfrak{p}}(s_{\alpha}.\lambda )
+\sum_{\alpha \in \Psi_{\lambda, {\rm iso}}^{i}}
 \chi^{\mathfrak{p}}_{\alpha}(\lambda-\alpha)=0 \;.
\end{empheq}
\end{thm}

%----------------------------------------------------------------------------------------------------
\begin{thm}[Irreducibility Criterion (2)]\label{thm.simple}
A parabolic Verma module $M_{\mathfrak{p}}(\lambda)$ is irreducible
if and only if 
\begin{empheq}[box={\mybluebox[7pt]}]{equation}
\label{eq.thm.simple}
\sum_{\alpha \in \Psi_{\lambda, {\rm non\mathchar`-iso}}}
 \chi^{\mathfrak{p}}(s_{\alpha}.\lambda )
+\sum_{\alpha \in \Psi_{\lambda, {\rm iso}}}
 \chi^{\mathfrak{p}}_{\alpha}(\lambda-\alpha)=0 \;.
\end{empheq}
\end{thm}

%----------------------------------------------------------------------------------------------------

\begin{proof}
Let us note that 
\beq
\Psi_{\lambda} = \bigsqcup_{i=1}^n \Psi^{i}_{\lambda}  \;.
\eeq
The sum in \eqref{eq.thm.simple} therefore can be decomposed into 
 the sum of \eqref{eq.prop_1} over $i$.
Since each summand is $v(i,-p_i(\lambda))$,
 whose coefficients are non-negative,
 the vanishing of the sum \eqref{eq.thm.simple} is equivalent to the 
 vanishing of \eqref{eq.prop_1} for all $i$.
\end{proof}

%----------------------------------------------------------------------------------------------------

The irreducibility
 criteria in Theorem \ref{prop_1} and Theorem \ref{thm.simple}
involve formal characters $\chi^{\frakp}$ and $\chi^{\frakp}_{\alpha}$,
 which take slightly involved expressions for our practical applications.
Fortunately we can simplify these conditions.

%----------------------------------------------------------------------------------------------------

\begin{lem}\label{lem.empty}
Suppose that $\Psi_{\lambda, {\rm iso}} \ne \varnothing$.
Then \eqref{eq.thm.simple} never holds,
 i.e.\ the parabolic Verma module $M_{\frakp}(\lambda)$ is reducible.
\end{lem}

%----------------------------------------------------------------------------------------------------

\begin{proof}
Suppose that the parabolic Verma module $M_{\frakp}(\lambda)$
is reducible. From Theorem \ref{thm.simple}
we have the relation \eqref{eq.thm.simple},
which when written in terms of Verma module characters reads
\beqn
\label{eq.simplicity_sum}
&\sum_{\alpha \in \Psi_{\lambda, {\rm non\mathchar`-iso}}} 
\sum_{w\in W_{\fl}} \det (w) {\rm ch}\, M(w.s_{\alpha}.\lambda ) \\
&\quad\quad +
\sum_{\alpha \in \Psi_{\lambda, {\rm iso}}} 
\sum_{w\in W_{\fl}} \det (w) 
\sum_{n=0}^{\infty}
(-1)^n {\rm ch}\, M(w.(\lambda-(n+1) \alpha)) =0 \;,
\eeqn
where we used \eqref{chi_def}.

For a root $\alpha \in \Psi_{\lambda, {\rm iso}}$ and
 an element of the Weyl group $w\in W_{\fl}$,
 there are infinitely many terms of the form
\beqn
\label{eq.infinite}
{\rm ch}\, M(w.(\lambda-(n+1)\alpha)) \;, \quad n=1,2, \ldots 
\eeqn
in \eqref{eq.simplicity_sum}.
Then each character of the form \eqref{eq.infinite}
 has to be canceled by some other Verma module characters
 (see Lemma \ref{lem.cancel}).
However the first term in \eqref{eq.simplicity_sum}
 contains only finitely many terms,
 and consequently still leaves infinitely many terms of the
 form \eqref{eq.infinite}.
Now characters inside the second term all take the form \eqref{eq.infinite}.
Since both $\Psi_{\lambda, {\rm iso}}$ and $W_{\fl}$ 
 are finite sets,
 we can find distinct pairs $(\alpha, w)$ and $(\alpha', w')$ such that 
 the terms in \eqref{eq.simplicity_sum} for them 
 contain at least two (in fact, infinitely many) characters of the same form,
 namely, there exist $n_1, n_2, n_1', n_2' \in \mathbb{N}$ such that
\beq
w.(\lambda-(n_i+1) \alpha)=w'.(\lambda-(n_i'+1) \alpha') \;, \quad i=1,2 \;.
\eeq
This means 
\beqn
\label{eq.lwl}
\lambda-(n_i+1) \alpha=w''.(\lambda-(n_i'+1) \alpha') \;, \quad i=1,2 \;.
\eeqn
for $w'':=w^{-1}w'\in W_{\fl}$, and by taking differences of two equations
we obtain
\beq
(n_1-n_2)\alpha=(n'_1-n'_2)w''\alpha' \;, \quad 
 \text{ i.e. $\alpha$ is proportional to $w''\alpha'$}  \;.
\eeq
We then have from \eqref{eq.lwl} that $\lambda-w''. \lambda \in \bC \alpha$.
Since $\lambda-w''.\lambda \in Q(\Delta_{\fl})$
 and $\alpha \in \bDelta_{\sub{1}}$, we must have $\lambda=w''.\lambda$. 
Then $w''=e$ because $\lambda+\rho$ is regular for $\Delta_{\fl}$.
Hence $w=w'$.
This shows that $\alpha$ and $\alpha'$ are proportional
 but since $\alpha,\alpha'\in\Delta_{\sub{1}}$ we have $\alpha=\alpha'$,
 which contradicts to our choice $(\alpha,w)\neq (\alpha',w')$.  
\end{proof}

%----------------------------------------------------------------------------------------------------

Thanks to Lemma \ref{lem.empty} we can concentrate on the case 
$\Psi_{\lambda, {\rm iso}} = \varnothing$.
Then the condition for irreducibility \eqref{eq.thm.simple}
now simplifies to 
\beqn
\label{eq.thm.simple.2}
\sum_{\alpha \in \Psi_{\lambda, {\rm non\mathchar`-iso}}}
 \chi^{\mathfrak{p}}(s_{\alpha}.\lambda )=0 \;.
\eeqn

%----------------------------------------------------------------------------------------------------

\begin{lem}\label{lem.chip}
Let $\lambda,\mu\in P(\Delta_{\fl})$.
Suppose that we have
\beqn\label{sum.chip}
\chi^{\frakp}(\mu) + \sum_{\nu\leq \lambda} c_{\nu}\chi^{\frakp}(\nu) =0
\eeqn
for some $c_{\nu}\in \bZ$. 
We then have at least one of the following two:
\begin{enumerate}
\item We have $\mu=w.\mu$ for some $w(\neq e)\in W_{\fl}$,
 in which case $\chi^{\frakp}(\mu)=0$. 
\item There exists $\nu$ such that $c_{\nu}\neq 0$ and 
 $\mu=w.\nu$ for some $w\in W_{\fl}$, in 
 which case $\chi^{\frakp}(\mu)=\pm \chi^{\frakp}(\nu)$.
\end{enumerate}
\end{lem}

%----------------------------------------------------------------------------------------------------

\begin{proof}
Recall that $\chi^{\frakp}(\mu)$
 is an alternating sum of $\textrm{ch}\, M(w.\mu)$
 over $w\in W_{\fl}$ (see \eqref{chi_def}). 
The assumption \eqref{sum.chip} can then be written as
\beqn\label{sum.chip.2}
\sum_{w\in W_{\fl}} \det (w) \, {\rm ch}\, M(w.\mu)
 + \sum_{\nu\leq \lambda}
 \sum_{w\in W_{\fl}} \det (w) c_{\nu} \, {\rm ch}\, M(w.\nu)  =0 \;.
\eeqn
Now from Lemma \ref{lem.cancel} the Verma module character ${\rm ch}\, M(\mu)$
 has to be canceled by another 
Verma module character, either from the first or the second term of the sum \eqref{sum.chip.2}.
When the canceling term comes from the first term,
 we have ${\rm ch}\, M(\mu)+{\rm ch}\, M(w.\mu)=0$
 for some $w(\neq e)\in W_{\fl}$, giving the first option.
If it comes from the second term,
 we have ${\rm ch}\, M(\mu)={\rm ch}\, M(w.\nu)$
 for some $w\in W_{\fl}$ and $\nu$ such that $c_\nu\ne 0$,
 giving the second option.
\end{proof}

%----------------------------------------------------------------------------------------------------

\begin{lem}\label{lem.ortho}
The identity \eqref{eq.thm.simple.2} implies that 
 for all $\alpha\in \Psi_{\lambda, {\rm non\mathchar`-iso}}$ there exists 
 $\beta\in \Delta^{i(\alpha)}\cap\Delta_{\sub{0}}$ such that 
 $(\lambda+\rho, \beta)=0$.
\end{lem}

%----------------------------------------------------------------------------------------------------

\begin{proof}
Let us fix an element $\alpha\in \Psi_{\lambda, {\rm non\mathchar`-iso}}$.
Since the sum
\eqref{eq.thm.simple.2} contains the character 
$\chi^{\frakp}(s_{\alpha}.\lambda)$
one of the two possibilities of Lemma \ref{lem.chip} happens.

Suppose we have the first option.
We then have $w.(s_{\alpha}.\lambda) =s_{\alpha}.\lambda$
for some $w=s_{\gamma}\in W_{\fl}$ with some $\gamma\in \Delta_{\fl}$.
This means 
\beq
s_{s_{\alpha}(\gamma)}.\lambda= (s_{\alpha} s_{\gamma}  s_{\alpha}).\lambda =
s_{\alpha}.s_{\gamma}.s_{\alpha}.\lambda = \lambda \;.
\eeq
By putting
 $\beta=s_{\alpha}(\gamma)\in \Delta^{i(\alpha)}\cap\Delta_{\sub{0}}$,
 we obtain $s_{\beta}. \lambda =\lambda$, i.e.\ $(\lambda+\rho, \beta)=0$.

Suppose we have the second option.
We then have certain $\gamma \in \Psi_{\lambda, {\rm non\mathchar`-iso}}$ 
 with $\gamma\ne \alpha$
 such that there exists $w(\neq e) \in W_{\fl}$
 with $s_{\alpha}.\lambda = w.(s_{\gamma}.\lambda)$.
We thus have $(s_{\alpha} w s_{\gamma}).\lambda = \lambda$.
This means that $\lambda+\rho$ lies on a wall of a Weyl chamber for
 $\Delta^{i(\alpha)}\cap \Delta_{\sub{0}}$,
 and hence $\lambda+\rho$ is fixed by a root reflection $s_{\alpha}$
 for a certain
 $\alpha\in \Delta^{i(\alpha)}\cap\Delta_{\sub{0}}$, as desired.
\end{proof}

%----------------------------------------------------------------------------------------------------

By collecting Theorem \ref{thm.simple}, Lemma \ref{lem.empty} and Lemma \ref{lem.ortho} 
we obtain the following proposition:

%----------------------------------------------------------------------------------------------------

\begin{prop}\label{prop.M}
If a parabolic Verma module $M_{\mathfrak{p}}(\lambda)$ is irreducible, then 
the following conditions hold:
\begin{itemize}

\item[(M):]
\begin{itemize}
\item
For all $\alpha\in \Psi_{\lambda, {\rm non\mathchar`-iso}}$ there exists 
$\beta\in \Delta^{i(\alpha)}\cap\Delta_{\sub{0}}$ such that 
$(\lambda+\rho,\beta)=0$.
\item 
$\Psi_{\lambda, {\rm iso}} = \varnothing$.
\end{itemize}
\end{itemize}
\end{prop}

The condition (M) in Proposition \ref{prop.M} is a necessary condition for
 the irreducibility, but not a sufficient condition.
Interestingly, a slightly stronger condition (M+) below
 turns out to be sufficient for the irreducibility:

%----------------------------------------------------------------------------------------------------

\begin{prop}\label{prop.Mp}
A parabolic Verma module $M_{\mathfrak{p}}(\lambda)$ is irreducible
 if the following conditions hold:
\begin{itemize}
\item[(M+):]
\begin{itemize}
\item
For all $\alpha\in \Psi_{\lambda, {\rm non\mathchar`-iso}}$ there exists 
$\beta\in \Delta^{i(\alpha)}\cap\Delta_{\sub{0}}$ such that 
$(\lambda+\rho, \beta)=0$ and 
$s_{\alpha}(\beta) \in \Delta_{\fl}$.
\item 
$\Psi_{\lambda, {\rm iso}} = \varnothing$.
\end{itemize}
\end{itemize}
\end{prop}
%----------------------------------------------------------------------------------------------------

\begin{proof}
Let us choose any root $\alpha\in \Psi_{\lambda, {\rm non\mathchar`-iso}}$.
Taking $\beta$ as in the condition we have 
\beq
s_{s_{\alpha}(\beta)}.( s_{\alpha}.\lambda)= 
s_{\alpha}.s_{\beta}. \lambda =s_{\alpha}.\lambda \;.
\eeq
Since $s_{\alpha} (\beta) \in \Delta_{\fl}$ from the assumption,
this means
we have an element $w(\neq e)\in W_{\fl}$ such that
 $w.(s_{\alpha}.\lambda)=s_{\alpha}.\lambda$.
Hence $\chi^{\mathfrak{p}}(s_{\alpha}.\lambda)=0$. 

Since we have shown that each summand in \eqref{eq.thm.simple} vanishes,
 the equation \eqref{eq.thm.simple} holds
 and thus thanks to Theorem \ref{thm.simple}
the parabolic Verma module $M_{\frakp}(\lambda)$ is irreducible.
\end{proof}

%----------------------------------------------------------------------------------------------------

The conditions (M) and (M+) are simpler than the 
 conditions \eqref{eq.prop_1} or \eqref{eq.thm.simple}.
Unfortunately, (M) is necessary but not sufficient, while (M+)
 is sufficient but not necessary, for irreducibility.
Indeed, there are examples
 where (M) is satisfied but (M+) is not (\cite{Jantzen}).
However, if we impose a regularity assumption on $\lambda$
 we can present a condition which is both necessary and sufficient.

%----------------------------------------------------------------------------------------------------

\begin{cor}\label{cor.Msharp}

Suppose that $(\lambda+\rho, \alpha)\ne 0$
 for all $\alpha\in \Delta_{\sub{0}}$.
Then the parabolic Verma module
$M_{\mathfrak{p}}(\lambda)$ is irreducible if and only if
the following condition hold:

\begin{itemize}
\item[(M++)]

$\Psi_{\lambda, {\rm non\mathchar`-iso}}=\Psi_{\lambda, {\rm iso}}=\varnothing$.

\end{itemize}

\end{cor}
%----------------------------------------------------------------------------------------------------

\begin{proof}
These follow easily from Propositions \ref{prop.M} and \ref{prop.Mp}.
\end{proof}

%%%%%%%%%%%%%%%%%%%%%%%%%%%%%%%%%%%%%%%%%%%%%%%%%%%%%%%%%%%%
%%%%%%%%%%%%%%%%%%%%%%%%%%%%%%%%%%%%%%%%%%%%%%%%%%%%%%%%%%%%

\bibliographystyle{nb}
\bibliography{bootstrap_super}

\end{document}